\documentclass[letterpaper,12pt]{article}

\usepackage{tabularx} 
\usepackage{amsmath,amsfonts,stmaryrd,amssymb,bm,mathrsfs, dsfont}  
\usepackage{graphicx} 
\usepackage{natbib}
\usepackage{subcaption}
\usepackage{algorithm,algorithmic}
\usepackage{booktabs}
\usepackage{multirow}
\usepackage{authblk}
\usepackage{paralist}
\usepackage{amsthm} 
\usepackage[margin=1in,letterpaper]{geometry} 
\usepackage[final]{hyperref} 
\usepackage{verbatim}
\hypersetup{
	colorlinks=true,       
	linkcolor=cyan,        
	citecolor=cyan,        
	filecolor=magenta,     
	urlcolor=cyan         
}
\usepackage{blindtext}
\allowdisplaybreaks[4]
 
\newtheorem{myThm}{Theorem}
\newtheorem{myLe}{Lemma}
\newtheorem{myA}{Assumption}
\newtheorem{myRmk}{Remark}
\newtheorem{myProp}{Proposition}

\begin{document}
\title{Spectral co-Clustering in Multi-layer Directed Networks}

\author[1]{\large Wenqing Su}
\author[2]{\large \;Xiao Guo{\linespread{1.4}\selectfont\footnote{{\it Corresponding authors.} Xiao Guo, xiaoguo@nwu.edu.cn. Ying Yang, yangying@tsinghua.edu.cn.}}} 
\author[3]{\large \;Xiangyu Chang}
\author[1]{\large Ying Yang$^{\tiny*}$}
\affil[1]{\normalsize Department of Mathematical Sciences, Tsinghua University}
\affil[2]{\normalsize School of Mathematics, Northwest University}
\affil[3]{\normalsize School of Management, Xi’an Jiaotong University}
\date{}  
\maketitle
\vspace{-10mm}

\begin{abstract}
Modern network analysis often involves multi-layer network data in which the nodes are aligned, and the edges on each layer represent one of the multiple relations among the nodes. Current literature on multi-layer network data is mostly limited to undirected relations. However, direct relations are more common and may introduce extra information. This study focuses on community detection (or clustering) in multi-layer directed networks. To take into account the asymmetry, a novel spectral-co-clustering-based algorithm is developed to detect \emph{co-clusters}, which capture the sending patterns and receiving patterns of nodes, respectively. Specifically, the eigendecomposition of the \emph{debiased} sum of Gram matrices over the layer-wise adjacency matrices is computed, followed by the $k$-means, where the sum of Gram matrices is used to avoid possible cancellation of clusters caused by direct summation. Theoretical analysis of the algorithm under the multi-layer stochastic co-block model is provided, where the common assumption that the cluster number is coupled with the rank of the model is relaxed. After a systematic analysis of the eigenvectors of the population version algorithm, the misclassification rates are derived, which show that multi-layers would bring benefits to the clustering performance. The experimental results of simulated data corroborate the theoretical predictions, and the analysis of a real-world trade network dataset provides interpretable results.

\end{abstract}

\noindent
{\it Keywords:} Multi-layer directed networks, Co-clustering, Spectral methods, Bias-correction

\section{Introduction}
Multi-layer network data arise naturally among various domains, where the nodes are the entities of interest and each network layer represents one of the multiple relations of the entities \citep{mucha2010community,holme2012temporal,kivela2014multilayer,boccaletti2014structure}. A specific type of multi-layer network features identical sets of nodes across each layer, with no inter-layer connections. In this work, we refer to such networks as multi-layer networks, as in \citet{paul2020spectral, jing2021community}, though they are also called multiplex networks. For example, the gene co-expression multi-layer network consists of genes co-expressed at different developmental stages of the animal, and the gene expression patterns at different stages may differ yet remain highly correlated \citep{bakken2016comprehensive, zhang2017finding}. The above example involves undirected relations. In real-world networks, more common relations are \emph{directed}. For example, the Worldwide Food and Agricultural Trade (WFAT) multi-layer network consists of the {direct} trade relationships among the same set of countries across different commodities. Trade relationships concerning different commodities differ but are not entirely unrelated \citep{de2015structural}. 

The multi-layer network has received considerable attention recently, see, e.g., \cite{della2020symmetries}; \cite{macdonald2022latent}; \cite{huang2022spectral} and references therein. Of particular interest is the community detection or clustering problem, where the goal is to partition the network nodes into disjoint \emph{communities} or \emph{clusters} with the help of multiple layers. We focus on the case where the underlying communities are consistent among all the layers. The community detection of multi-layer networks has been well-studied in the lens of multi-layer stochastic block models (multi-layer SBMs) \citep{han2015consistent, valles2016multilayer, paul2016consistent}, namely, equipping each layer of the multi-layer network with a stochastic block model (SBM) \citep{holland1983stochastic}. In an SBM, nodes are first partitioned into disjoint communities, and based on the community membership, the nodes are linked with probability specified by a \emph{block  probability matrix}. Specifically, \cite{han2015consistent} studied the asymptotic properties of spectral clustering and maximum likelihood estimation for multi-layer SBMs as the number of layers increases and the number of nodes remains fixed. \cite{paul2020spectral} studied several spectral and matrix factorization-based methods and provided theoretical guarantees under multi-layer SBMs. \cite{lei2020consistent} derived consistent results for a least squares estimation of community memberships and proved consistency of the global optima for general block structures without imposing the positive-semidefinite assumption for individual layers.
\cite{LeiJ2022Bias} proposed a bias-adjusted spectral clustering for multi-layer SBMs and derived a novel aggregation strategy to avoid the community cancellation of different layers. Considering that different layers may have different community structures, \cite{jing2021community} introduced a mixture multi-layer SBM and proposed a tensor-based method to reveal both memberships of layers and memberships of nodes. \cite{wang2021fast} and \cite{fu2023profile} studied the recovery of community structures in multi-layer SBMs using the pseudo-likelihood method. Also see \citet{bhattacharyya2018spectral,pensky2019spectral,arroyo2021inference,noroozi2022sparse}; among others.

Despite the great efforts on the community detection of multi-layer networks, the following two issues remain to be tackled. First, existing literature is mostly limited to undirected networks. For directed networks, the most common approach is to simply ignore edge directions and use the methods developed for multi-layer undirected networks. However, this simplistic technique is unsatisfactory since the potentially useful information contained in edge directions is not retained. For example, in the WFAT multi-layer network, the trade transactions among countries include both import and export, which is quite different for even a single country. Therefore, clustering nodes regardless of their edge directions would be coarse. To incorporate the asymmetry of directed networks, instead of partitioning the nodes into one set of clusters, it is more reasonable to \emph{co-cluster} the network nodes \citep{malliaros2013clustering,rohe2016co,zhang2022directed}. That is, clustering the nodes in two ways to obtain \emph{co-clusters}, namely, the row clusters (or the sending clusters) and the column clusters (or the receiving clusters). The nodes in the same row (column) cluster have similar sending (receiving) patterns. Hence, the community detection of multi-layer directed networks should be in the context of co-clustering.  

Second, most existing literature on SBMs assumes that the population matrix has rank coupled with the number of underlying communities. That is, the block probability matrix is assumed to be of \emph{full} rank. Such assumption brings benefits to the algebraic properties of the matrices arising from SBMs. However, this assumption is not necessarily met by practical tasks. In the context of multi-layer SBMs, this assumption is inherited in that the population matrix, for example, the \emph{sum} of each layer-wise SBM, is commonly assumed to be of rank equaling  the number of communities, although each \emph{layer-wise} SBM does not necessarily satisfy the assumption \citep{paul2020spectral}. 
In the context of single-layer SBMs, \cite{tang2022asymptotically} studied the asymptotic normality of the estimated block probability matrix without the full rank assumption, while their asymptotic theoretical machinery is not directly applicable to our non-asymptotic justification of the clustering performance.
Thus, to meet practical needs, it is desirable to study the theoretical properties of the population model of multi-layer directed  networks in the rank-deficient regime.  

Motivated by the above problems, we study the problem of co-clustering the multi-layer directed network. We typically assume that each layer of the network is generated from a stochastic co-block model (ScBM) \citep{rohe2016co}, where the underlying row clusters and column clusters are not necessarily the same. Following convention, we call the model over all layers the \emph{multi-layer} ScBM. The first contribution of this paper is a neat, flexible, and computationally efficient spectral-clustering-based algorithm for co-clustering the multi-layer directed network. To avoid possible cancellation of clusters among layers, we use the \emph{sum of Gram} (SoG) matrices of the row (column) spaces of the layer-wise adjacency matrices as the algorithm's target matrix in order to detect row (column) clusters. Then, we perform a bias-correction on the two-way SoG matrices to remove the bias that comes from their diagonal entries. After that, we apply the spectral clustering to the two debiased SoG matrices with possibly different cluster numbers. As the leading eigenvectors approximate the row and column spaces of the two debiased SoG matrices, respectively, it is expected that the resulting two sets of clusters contain nodes with similar sending and receiving patterns, respectively. To the best of our knowledge, this is the first work to study the community detection problem for multi-layer directed networks. 

The second contribution of this work is that we systematically study the algebraic properties of the population version of SoG matrices whose block  probability matrix is \emph{not} necessarily of full rank. In particular, we provide interpretable conditions under which the eigendecomposition (i.e., the first step of the spectral clustering) of the population version of SoG matrices would reveal the underlying communities in multi-layer ScBMs. Based on these findings, we provide rigorous analysis of the consistency of the community estimates. Specifically, we use the decoupling techniques \citep{de1995decoupling,LeiJ2022Bias} to derive the concentration inequalities of the sum of quadratic asymmetric matrices. We use the derived inequalities to bound the misclassification rate of the row and column clusters, respectively.

The remainder of the paper is organized as follows. Section \ref{model} presents the model for multi-layer directed networks and investigates its algebraic properties. Section \ref{methodandtheory} develops the debiased spectral co-clustering algorithm and proves its consistency. Section \ref{simulation} illustrates the finite sample performance of the proposed method via simulations. Section \ref{realdata} includes a real-world application of the proposed method to the WFAT dataset. Section \ref{conclusion} concludes the paper and provides possible extensions. Technical proofs are included in the Appendix.

\section{The multi-layer ScBM and its algebraic properties}\label{model}
In this section, we first present the multi-layer ScBM for modeling the multi-layer directed network. Next, we study its algebraic properties for understanding the population-wise clustering behavior of spectral co-clustering based on the SoG matrices.

\emph{Notes and notation:}
We use $[n]$ to denote the set $\{1,...,n\}$. For a matrix $M\in \mathbb R^{m\times n}$ and index $i\in [m]$ and $j\in [n]$, $M_{i\ast}$ and $M_{\ast j}$ denote the $i$th row and $j$th column of $M$, respectively. $\|M\|_F$,  $\|M\|_{\max}$, $\|M\|_{1, \infty}$, and $\|M\|_{2, \infty}$ denote the Frobenius norm, the element-wise maximum absolute value, the maximum row-wise $l_1$ norm and the maximum row-wise $l_2$ norm of a given matrix $M$, respectively. In addition, $\|\cdot\|_2$ denotes the Euclidean norm of a vector or the spectral norm of a matrix. We will use $c,c_0,c_1$, and more generally $c_i$, to denote constants that vary across different contexts.
Following convention, we will use clusters and communities interchangeably.

\subsection{Multi-layer ScBMs}
\label{multi-layer}
Consider the multi-layer directed network with $L$-layers and $n$ common nodes, whose adjacency matrices are denoted by $A_l$, where $A_l \in \{0, 1\}^{n\times n}$ for all $1 \leq l \leq L$. We assume that all the layers share common row and column clusters but with possibly different edge densities.  
In particular, suppose that $n$ nodes are assigned to $K_y$ non-overlapping row clusters and $K_z$ non-overlapping column clusters, respectively. The number of nodes in the row (column) cluster $k\in [K_y]$ ($k\in [K_z]$) is denoted by $n_k^y$ ($n_k^z$). For $i\in[n]$, the row (column) cluster assignment of node $i$ is given by $g_i^y\in [K_y]$ ($g_i^z\in[K_z]$). Given the cluster assignment, we assume the layer-wise networks are generated independently from the following ScBM \citep{rohe2016co}. That is, for any pair of nodes $i \neq j$ (with $j\in[n]$) and any layer $l\in[L]$, each $A_{l,ij}$ is generated independently according to
\begin{equation}
\label{adj_generate}
A_{l,ij}\sim {\rm Bernoulli}(\rho B_{l,g_i^y g_j^z}),
\end{equation}
where $\rho\in(0,1]$ is an overall edge density parameter, $B_l\in \mathbb [0,1]^{K_y\times K_z}$ denotes the \emph{heterogeneous} block  probability matrix indicating the community-wise edge probabilities in each $l$. While for any $l\in [L]$ and $i=j$, $A_{l,ii}=0$. It can be seen from \eqref{adj_generate} that nodes in a common row (column) cluster are stochastically equivalent senders (receivers) in the sense that they send out (receive) an edge to a third node with equal probabilities. Putting together the $L$ layer-wise networks $\{A_l\}_{l=1}^L$, we say that the multi-layer network is generated from the multi-layer ScBM.

Throughout this paper, we assume that the number of communities is fixed and the community sizes are balanced. Specifically, we make the following Assumption \ref{asmp1}.

\begin{myA}\label{asmp1}
Both the number of row clusters $K_y$ and the number of column clusters $K_z$ are fixed. The community sizes are balanced, that is, there exists a constant $c_0\geq1$ such that each row cluster size is in $[c_0^{-1}n/K_y, c_0n/K_y]$ and each column cluster size is in $[c_0^{-1}n/K_z, c_0n/K_z]$.
\end{myA}

\subsection{Algebraic properties of multi-layer ScBMs}
It is essential to investigate the algebraic properties of multi-layer ScBMs in order to understand the rationality of a clustering algorithm from the angle of population. Before that, we must specify the clustering algorithm.     

For the single-layer ScBM, \emph{spectral co-clustering} \citep{rohe2016co,guo2023randomized} is a popular and effective algorithm, which first computes the singular value decomposition (SVD) of a matrix, say the adjacency matrix $A$, and then implements $k$-means on the left and right singular vectors to obtain the row and column clusters, respectively. For the multi-layer ScBM, we also proceed to develop the spectral co-clustering-based method to detect the co-clusters. It is natural to use the summation matrix $\sum_{l=1}^LA_l$ as the input of spectral co-clustering. However, such direct summation may lead to cancellation of clusters. For example, suppose $B_1:= \begin{bmatrix} a & b \\ c & d  \end{bmatrix}$ and $B_2:= \begin{bmatrix} b & a \\ d & c \end{bmatrix}$, the sum $B_1 + B_2 = \begin{bmatrix} a+b & a+b\\ c+d& c+d  \end{bmatrix}$ results in identical columns and thus provides no signal for the column clusters. 
Considering that the $ij$th entry of $A_lA_l^T$ (resp. $A_l^TA_l$) counts the number of common children (resp. parents) nodes of nodes $i$ and $j$, we proceed to use the leading eigenvectors of SoG matrices $\sum_{l=1}^LA_lA_l^T$ (resp. $\sum_{l=1}^LA_l^TA_l$) as the input of subsequent $k$-means clustering, in order to obtain the row (resp. column) clusters of nodes with similar sending (resp. receiving) patterns. We will modify the algorithm in the next section.

In the sequel, we investigate the theoretical properties of the eigenvectors of the population version of $\sum_{l=1}^LA_lA_l^T$ and $\sum_{l=1}^LA_l^TA_l$. Before that, we give some notations. Let $Y\in \{0,1\}^{n\times K_y}$ and $Z\in \{0,1\}^{n\times K_z}$ be the row and column membership matrices, respectively, where each row consists of all 0's except one 1. In particular, $Y_{ig_i^y}=1$ and $Z_{ig_i^z}=1$ for each $i$. For $l\in[L]$, denote $\mathcal P_l=\rho YB_lZ^T \in [0,1]^{n \times n}$, it is easy to see that $\mathcal P_l$ serve as the population matrices for $A_l$, in the sense that $\mathcal P_l-{\rm diag}(\mathcal P_l)=\mathbb E(A_l)$. The subsequent lemmas indicate that the rows of the eigenvectors of $\sum_{l=1}^L \mathcal P_l\mathcal P_l^T$ and $\sum_{l=1}^L \mathcal P_l^T\mathcal P_l$ could reveal the true row and column clusters, respectively. The proofs of the lemmas are provided in Appendix \ref{appendixB}. We begin by considering the row clusters.

\begin{myLe}\label{row interpretation}
Consider the multi-layer ScBM parameterized by $(Y_{n\times K_y},Z_{n\times K_z},\rho \{B_l\}^L_{l=1})$. Suppose {\rm rank($\sum_{l=1}^LB_lB_l^T$)} = $K(K \leq K_y)$. Denote the eigendecomposition of $\sum_{l=1}^L \mathcal P_l\mathcal P_l^T$ by $U\Lambda^R U^T$, where $U$ is an $n\times K$ matrix with orthonormal columns and $\Lambda^R$ is a $K\times K$ diagonal matrix. Denote the eigendecomposition of $\Delta_y \sum_{l=1}^L B_l\Delta_z^2B_l^T \Delta_y$ by $Q^RD^R{Q^R}^T$, where $\Delta_y: = {\rm diag}(\sqrt{n_1^y}, \ldots, \sqrt{n_{K_y}^y})$ and $\Delta_z := {\rm diag}(\sqrt{n_1^z}, \ldots, \sqrt{n_{K_z}^z})$. Then we have 
	
	(a) If $\sum_{l=1}^LB_lB_l^T$ is of full rank, i.e., $K=K_y$, then $Y_{i\ast} = Y_{j\ast}$ if and only if $U_{i\ast} = U_{j\ast}$. Otherwise, for any $Y_{i\ast} \neq Y_{j\ast}$, we have $\|U_{i\ast} - U_{j\ast}\|_2 = \sqrt{n_{g_i^y}^{-1} + n_{g_j^y}^{-1}}$.
	
	(b) If $\sum_{l=1}^LB_lB_l^T$ is rank-deficient, i.e., $K<K_y$, then for $Y_{i\ast} = Y_{j\ast}$, we have $U_{i\ast} = U_{j\ast}$. Otherwise, if $\Delta_y^{-1}Q^R$ has mutually distinct rows and there exists a deterministic positive sequence $\{\zeta_n\}_{n\geq 1}$ such that
	\begin{equation}\label{row separable}
		\min_{1 \leq k \neq k' \leq K_y} \|\frac{Q^R_{k\ast}}{\sqrt{n_{k}}} - \frac{Q^R_{k'\ast}}{\sqrt{n_{k'}}}\|_2 \geq \zeta_n,
	\end{equation}
	then for any $Y_{i\ast} \neq Y_{j\ast}$, we have $\|U_{i\ast} - U_{j\ast}\|_2 \geq \zeta_n$. 
\end{myLe}

\begin{myRmk}
    Note that in the literature on ScBM, the row cluster number $K_y$, the column cluster number $K_z$ and the rank $K$ of population matrix are commonly coupled, say, it is often assumed that $K=K_y\leq K_z$ or $K=K_z\leq K_y$ \citep{rohe2016co}. We here relax this assumption and make the model more practical. 
\end{myRmk}

Lemma \ref{row interpretation} shows that when two nodes are in the same row cluster, the corresponding rows of $U$ coincide. Conversely, if the nodes do not belong to the same row cluster, a gap is present between their corresponding rows in $U$. As we will see, this brings confidence to the success of the spectral co-clustering using the sample version of $\sum_{l=1}^L \mathcal P_l\mathcal P_l^T$. Note that in the rank-deficient case, we additionally require \eqref{row separable} holds to ensure that two rows of $U$ are separable for nodes with different row clusters. 

It is desirable to study the sufficient and interpretable condition to satisfy \eqref{row separable}. Define a flattened block  probability matrix $B^R := [B_1, \dots, B_L]\in [0,1]^{K_y\times L K_z}$, where the $k$th row contains the overall sending pattern of the $k$th row cluster to each of the column clusters across all layers. The following lemma provides an explicit condition on $B^R$ such that (\ref{row separable}) holds. 
\begin{myLe}\label{row separable condition}
	Under the same multi-layer ScBM as in Lemma {\rm \ref{row interpretation}} and Assumption {\rm \ref{asmp1}}, if $B^R$ satisfies
	\begin{equation}
 \label{euclidean}
		\min_{1 \leq k \neq k' \leq K_y}  c_0^{-1} L^{-1}  \langle B^R_{k\ast}, B^R_{k\ast} \rangle  +  c_0^{-1} L^{-1}  \langle B^R_{k'\ast}, B^R_{k'\ast} \rangle  - 2c_0L^{-1}   \langle B^R_{k\ast}, B^R_{k'\ast} \rangle  \geq c_0^2n\zeta_n^2,
	\end{equation}
then condition \eqref{row separable} is met. Here $\zeta_n$ is the RHS of \eqref{row separable} and $c_0 \geq 1$ is the constant specified in Assumption {\rm \ref{asmp1}}.
\end{myLe}

Lemma \ref{row separable condition} is interpretable in that it requires certain difference between any two row pairs in $B^R$. To see this more clearly, if the column clusters are absolutely balanced, then it turns out that $c_0=1$ and the LHS of \eqref{euclidean} is equivalent to the Euclidean distance (divided by $L$) of two different rows of $B^R$.

\begin{myRmk}
It is worth mentioning that we only require the overall difference of each row cluster pair. Hence, some layers with weak cluster signal can borrow the strength from other layers with strong cluster signal, which shows the benefit of combining the layer-wise networks. 
\end{myRmk}

The aforementioned results focus on the row clusters. For the column clusters, we can derive similar results if we study instead $\sum_{l=1}^L \mathcal P_l^T\mathcal P_l$. 

\begin{myLe}\label{column interpretation}
	Under the same multi-layer ScBM as in Lemma {\rm \ref{row interpretation}} and suppose {\rm rank}$(\sum_{l=1}^LB_l^TB_l)$ = $K^\prime(K^\prime \leq K_z)$. Denote the eigendecomposition of $\sum_{l=1}^L \mathcal P_l^T\mathcal P_l$ by $V\Lambda^CV^T$, where $V$ is an $n\times K^\prime$ matrix with orthonormal columns and $\Lambda^C$ is a $K^\prime\times K^\prime$ diagonal matrix. Denote the eigendecomposition of $\Delta_z \sum_{l=1}^L B_l^T\Delta_y^2B_l \Delta_z$ by $Q^C D^C {Q^C}^T$. Then we have
	
		(a) If $\sum_{l=1}^LB_l^TB_l$ is of full rank, i.e., $K^\prime = K_z$, then $Z_{i\ast} = Z_{j\ast}$ if and only if $V_{i\ast} = V_{j\ast}$. Otherwise, for any $Z_{i\ast} \neq Z_{j\ast}$, we have $\|V_{i\ast} - V_{j\ast}\|_2 = \sqrt{n_{g_i^z}^{-1} + n_{g_j^z}^{-1}}$.
		
		(b) If $\sum_{l=1}^LB_l^TB_l$ is rank-deficient, i.e., $K^\prime <K_z$, then for $Z_{i\ast} = Z_{j\ast}$, we have $V_{i\ast} = V_{j\ast}$. Otherwise, if $\Delta_z^{-1}Q^C$ has mutually distinct rows and there exists a deterministic positive sequence $\{\xi_n\}_{n\geq 1}$ such that
	\begin{equation}\label{column separable}
		\min_{1 \leq k \neq k' \leq K_z} \|\frac{Q^C_{k\ast}}{\sqrt{n_{k}}} - \frac{Q^C_{k' \ast}}{\sqrt{n_{k'}}}\|_2 \geq \xi_n,
	\end{equation}
	then for any $Z_{i\ast} \neq Z_{j\ast}$, we have $\|V_{i\ast} - V_{j\ast}\|_2 \geq \xi_n$. 
\end{myLe}
Lemma \ref{column interpretation} shows that the leading eigenvectors of $\sum_{l=1}^L \mathcal P_l^T\mathcal P_l$ can expose the true underlying column clusters, where when $\sum_{l=1}^LB_l^TB_l$ is rank-deficient, we need extra condition (\ref{column separable}). In the following Lemma \ref{column separable condition}, we provide a sufficient condition under which (\ref{column separable}) holds. In particular, when $\sum_{l=1}^LB_l^TB_l$ is rank-deficient, we provide an interpretable condition on the flattened matrix $B^C := [B_1^T, \ldots, B_L^T] \in [0,1]^{K_z\times L K_y}$ which is sufficient for (\ref{column separable}). 

\begin{myLe}\label{column separable condition}
	Under the same multi-layer ScBM as in Lemma {\rm \ref{column interpretation}} and Assumption {\rm \ref{asmp1}}, if  $B^C$ satisfies
 	\begin{equation*}
		\min_{1 \leq k \neq k' \leq K_z} c_0^{-1} L^{-1}\langle B^C_{k\ast}, B^C_{k\ast} \rangle  + c_0^{-1} L^{-1}\langle B^C_{k'\ast}, B^C_{k'\ast}\rangle  - 2c_0L^{-1}\langle B^C_{k\ast}, B^C_{k'\ast} \rangle  \geq c_0^2n\xi_n^2,
	\end{equation*}
then \eqref{column separable} holds. Here $\xi_n$ is the RHS of \eqref{column separable} and $c_0 \geq 1$ is the constant specified in Assumption {\rm \ref{asmp1}}.
\end{myLe}

\section{Debiased spectral co-clustering and its consistency}
\label{methodandtheory}
In this section, we formally present the spectral-co-clustering algorithm based on the SoG matrices, where we  provide a bias-adjustment strategy to remove the bias of the SoG matrix. Then we study its theoretical properties in terms of misclassification rate. 

\subsection{Debiased spectral co-clustering}
We begin by considering the row clusters. 
As illustrated in Section \ref{model}, we have shown that $\sum_{l=1}^L A_lA_l^T$ is a good surrogate of $\sum_{l=1}^L A_l$ for avoiding possible row clustering cancellation, and we provide theoretical support that the population-wise matrix $\sum_{l=1}^L \mathcal P_l\mathcal P_l^T$ has eigenvectors that can reveal the true underlying row clusters, recalling that $\mathcal P_l -{\rm diag} (\mathcal P_l)= \mathbb E(A_l)$. However, similar to the undirected case studied in \citet{LeiJ2022Bias}, we will see that $\sum_{l=1}^L A_lA_l^T$ turns out to be a biased estimate of $\sum_{l=1}^L \mathcal P_l\mathcal P_l^T$. 

For notational simplicity, denote $\bar P_l := \mathcal P_l -{\rm diag} (\mathcal P_l)$ and denote $X_l := A_l - \bar P_l.$ Then we can decompose the deviation of $\sum_{l=1}^L A_lA_l^T$ from $\sum_{l=1}^L \mathcal P_l\mathcal P_l^T$ as
 \begin{eqnarray}\label{sumAA^T}
    \sum_{l=1}^L A_lA_l^T-\sum_{l=1}^L\mathcal P_l\mathcal P_l^T:=N_1+N_2+N_3+\mbox{diag}(\sum_{l=1}^L X_lX_l^T),
\end{eqnarray}
where 
\begin{align}
\label{decomposition}
N_1 & =  \sum_{l=1}^L\left(\mbox{diag}^2(\mathcal P_l) - \mathcal P_l\mbox{diag}(\mathcal P_l) - \mbox{diag}(\mathcal P_l)\mathcal P_l^T\right),  \nonumber \\
	N_2 & = \sum_{l=1}^L(\bar P_lX_l^T + X_l\bar P_l^T), \nonumber \\
	N_3 & =  \sum_{l=1}^LX_lX_l^T - \mbox{diag}(\sum_{l=1}^LX_lX_l^T).
\end{align}
The terms $N_1$, $N_2$ and $N_3$ are all relatively small. While for $\mbox{diag}(\sum_{l=1}^L X_lX_l^T)$, we have the following argument for its $i$th diagonal element, 
\begin{eqnarray}\label{S2ii}
	(\mbox{diag}(\sum_{l=1}^LX_lX_l^T))_{ii} & = & \sum_{l=1}^L\sum_{j=1}^n \bar{P}_{l,ij}^2\mathbb{I}(A_{l,ij}=0) + \sum_{l=1}^L\sum_{j=1}^n (1 - \bar{P}_{l,ij})^2 \mathbb{I}(A_{l,ij}=1) \nonumber\\
	& \leq & Ln \max_{l,ij}\bar{P}_{l,ij}^2 + \sum_{l=1}^Ld_{l,i}^{out},
\end{eqnarray}
where $d_{l,i}^{out}:=\sum_{j=1}^n A_{l,ij}$ is the out-degree of node $i$ in layer $l$. Note that the expectation of $\sum_{l=1}^L d_{l,i}^{out}$ is of the same order as $ Ln\max_{l,ij}\bar{P}_{l,ij}$, which dominates the first term $Ln \max_{l,ij}\bar{P}_{l,ij}^2$ under the sparse regime where $\bar{P}_{l,ij} = o(1)$. As a result, to reduce the upper bound, we can directly minus $\sum_{l=1}^Ld_{l,i}^{out}$ from both sides of \eqref{S2ii}. Specifically, define the row-wise bias-adjusted SoG matrix by
\begin{equation}\label{S^R}
	S^R = \sum_{l=1}^L(A_lA_l^T - D_l^{out}),
\end{equation}
where $D_l^{out} = \mbox{diag}(d_{l,1}^{out}, \ldots, d_{l,n}^{out})$. Then, the row clusters partition can be obtained by performing $k$-means on the row of the $K$ leading eigenvectors of $S^R$.

Similarly, for the column clusters, we can define the column-wise bias-adjusted SoG matrix by 
\begin{equation}\label{S^C}
	S^C = \sum_{l=1}^L(A_l^TA_l - D_l^{in}),
\end{equation}
where $D_l^{in} = \mbox{diag}(d_{l,1}^{in}, \ldots, d_{l,n}^{in})$ with $d_{l,j}^{in}:=\sum_{i=1}^n A_{l,ij}$. The column clusters partition can then be obtained by applying $k$-means on the row of the $K'$ leading eigenvectors of $S^C$.

We summarize the spectral co-clustering based on the \textsf{d}ebiased \textsf{SoG} matrices in Algorithm \ref{alg1}, and in what follows, we will refer to the algorithm by \textsf{DSoG}. 

\begin{algorithm}
	\renewcommand{\algorithmicrequire}{\textbf{Input:}}
	\renewcommand{\algorithmicensure}{\textbf{Output:}}
	\caption{Spectral co-clustering based on the \textbf{d}ebiased \textbf{s}um \textbf{o}f \textbf{G}ram matrices (DSoG)}
	\label{alg1}
	\begin{algorithmic}[1]
		\REQUIRE Adjacency matrices $A_1, \ldots, A_L$,  row cluster number $K_y$,  column cluster number $K_z$, target ranks $K(K\leq K_y)$ and $K'(K'\leq K_z)$;
		\ENSURE Estimated row membership matrix $\widehat{Y}$, and column membership matrix $\widehat{Z}$;
		
        \STATE Find the $K$ leading eigenvectors $\widehat{U}$ of $S^R$ in (\ref{S^R}), and the $K'$ leading eigenvectors $\widehat{V}$ of $S^C$ in (\ref{S^C}).
        \STATE Treat each row of $\widehat{U}$ as a point in $\mathbb{R}^{n\times K}$ and run the $k$-means with $K_y$ clusters.\\ Treat each row of $\widehat{V}$ as a point in $\mathbb{R}^{n\times K'}$ and run the $k$-means with $K_z$ clusters.
	\end{algorithmic}  
\end{algorithm}

\subsection{Consistency}
We measure the quality of a clustering algorithm by the misclassification rate. Specifically, for the row clusters, it is defined as
\begin{equation}\label{misclustering rate}
	\mathcal{L}(Y,\widehat{Y}) = \min_{\Psi \in \bm{\Psi}_{K_y}} \frac{1}{n}\|\widehat{Y}\Psi - Y\|_0,
\end{equation} 
where $\widehat{Y}\in \{0,1\}^{n\times K_y}$ and $Y$ correspond to the estimated and true membership matrices with respect to the row clusters, respectively, and $\bm{\Psi}_{K_y}$ is the set of all $K_y \times K_y$ permutation matrices. Similarly, we can define the misclassification rate with respect to the column clusters by $\mathcal{L}(Z,\widehat{Z})$ .

To establish theoretical bounds on the misclassification rates, we need to pose assumption on the $B_l$ matrices. We consider the rather general case where the aggregated squared block  probability matrices $\sum_{l=1}^LB_lB_l^T$ and $\sum_{l=1}^LB_l^TB_l$ are allowed to be rank-deficient, where only a linear growth of their minimum \emph{non-zero} eigenvalue is required. Specifically, we have the following Assumption \ref{asmp2}.

\begin{myA}\label{asmp2}
	The $K$th $(K \leq K_y)$ non-zero eigenvalue of $\sum_{l=1}^LB_lB_l^T$ and the $K^\prime$th $(K^\prime \leq K_z)$ {non-zero} eigenvalue of $\sum_{l=1}^LB_l^TB_l$ are at least $c_1L$ for some constant $c_1>0$.
\end{myA}

\begin{myRmk}
    Compared with literature on multi-layer SBMs, see for example \citet{arroyo2021inference,LeiJ2022Bias}, Assumption {\rm \ref{asmp2}} is much weaker. On the one hand, we do not require each $B_l$ being of full rank, which is the benefit of combining layer-wise networks. On the other hand, the combined block  probability matrix $\sum_{l=1}^LB_lB_l^T$ is also flexible to be degenerate. 
\end{myRmk}

The following theorem provides an upper bound on the proportion of misclustered nodes in terms of row clusters under the multi-layer ScBM mentioned in Lemma \ref{row interpretation}.

\begin{myThm}\label{row_thm}
Suppose that Assumptions {\rm \ref{asmp1}} and {\rm \ref{asmp2}}, and {\rm (\ref{row separable})} hold. If  $L^{1/2}n\rho \geq c_2\log(L+n)$ and $n\rho \leq c_3$ for positive constants $c_2$ and $c_3$, then the output 
$\widehat{Y}$ of Algorithm {\rm \ref{alg1}} satisfies  
	\begin{equation}\label{rowcluster_bound}
		\mathcal{L}(Y,\widehat{Y}) \leq \frac{c_4}{n{\zeta_n}^2}\left(\frac{1}{n^2} + \frac{\log(L+n)}{Ln^2\rho^2}\right)
	\end{equation}
with probability at least $1-O((L+n)^{-1})$ for some constant $c_4>0$, where recall that $\zeta_n$ appears in  \eqref{row separable}.
\end{myThm}

The proof of Theorem  \ref{row_thm} is given in Appendix \ref{appendixB}. Theorem \ref{row_thm} shows that in certain sense, the number of layers $L$ can boost the clustering performance. In addition, as expected, large $\zeta_n$ would bring benefit to the clustering performance. In particular, if $\sum_{l=1}^LB_lB_l^T$ is of full rank, the RHS of \eqref{rowcluster_bound} can be simplified to 
	\begin{equation*}
		c_5\left(\frac{1}{n^2} + \frac{\log(L+n)}{Ln^2\rho^2}\right)
	\end{equation*}
	for some constant $c_5>0$.
 
\begin{myRmk} 
Define the $K$th smallest non-zero eigenvalue of $\sum_{l=1}^LB_lB_l^T$ by $\sigma_K$. Without specifying the linear growth of $\sigma_K$ with $L$ as in Assumption \ref{asmp2}, the misclassification rate is bounded by
	\begin{equation*}
		\mathcal{L}(Y,\widehat{Y}) \leq \frac{c_6}{n{\zeta_n}^2}\left(\frac{L^2}{n^2\sigma_K^2} + \frac{L\log(L+n)}{\sigma_K^2 n^2\rho^2}\right).
	\end{equation*}
It should be noted that when choosing different growth rate of $\sigma_K$, the misclassification rate as long as the requirement of $\rho$ to obtain consistency would be modified accordingly. 
\end{myRmk}

\begin{myRmk}
Note that in (\ref{rowcluster_bound}), we can also replace $\zeta_n$ by $\min_{1 \leq k \neq k' \leq K_y} \|\frac{Q^R_{k\ast}}{\sqrt{n_{k}}} - \frac{Q^R_{k'\ast}}{\sqrt{n_{k'}}}\|_2$, namely, the LHS of (\ref{row separable}). We here use $\zeta_n$ to emphasize the lower bound of the row distance of eigenvector matrix. 
\end{myRmk}

\begin{myRmk}
    The balances of both the true row and true column clusters (i.e., Assumption \ref{asmp1}) are used to simplify the misclassification rate in Theorem {\rm\ref{row_thm}}. Specifically, Assumption {\rm \ref{asmp1}} is used to provide an explicit lower bound for the minimum non-zero eigenvalues of $\sum_{l=1}^L\mathcal P_l\mathcal P_l^T$; see details in the proof.
\end{myRmk}

In Theorem {\rm \ref{row_thm}}, we consider a situation where the layer-wise networks are rather sparse, with $n\rho \leq c_3$. This situation is interesting because in the single-layer network clustering problem, a well-known necessary condition for the weak consistency is $n\rho \geq c$ for some constant $c>0$. In the context of multi-layer networks, the increasing of network layers help to alleviate the requirement for $\rho$ by $\sqrt{L}$ (i.e., $L^{1/2}n\rho \geq c_2\log(L+n)$). In Proposition \ref{dense_thm} of Appendix \ref{AppendixC}, we also establish the misclassification rate under the situation $n\rho \geq c_7 \log (L+n)$. In addition, we show the effect of bias-adjustment in Proposition \ref{sog_thm} of Appendix \ref{AppendixC}.

Analogous to the row clusters, we provide the following results on the misclassification rate with respect to the column clusters.
\begin{myThm}\label{column_thm}
Suppose that Assumptions {\rm \ref{asmp1}} and {\rm \ref{asmp2}}, and {\rm \eqref{column separable}} hold. If $L^{1/2}n\rho \geq c_{8}\log(L+n)$ and $n\rho \leq c_{9}$ for positive constants $c_{8}$ and $c_{9}$, then the output $\widehat{Z}$ of Algorithm \ref{alg1} satisfies 
	\begin{equation*}
		\mathcal{L}(Z,\widehat{Z}) \leq \frac{c_{10}}{n{\xi_n}^2}\left(\frac{1}{n^2} + \frac{\log(L+n)}{Ln^2\rho^2  }\right)
	\end{equation*}
	with probability at least $1-O((L+n)^{-1})$ for some constant $c_{10}>0$, where we recall that $\xi_n$ appears in  {\rm \eqref{column separable}}.
\end{myThm}
As the proof of Theorem \ref{column_thm} follows a similar approach to that of Theorem \ref{row_thm}, it is omitted for brevity. Also, we can similarly discuss the results of Theorem \ref{column_thm}. 

\begin{myRmk}
     In the proofs of Theorems {\rm\ref{row_thm}} and {\rm\ref{column_thm}}, we assume that the $k$-means algorithm finds the optimal solution, while we use the heuristic Lloyd’s algorithm to solve $k$-means in experiments due to its efficiency and satisfactory empirical performance. Alternatively, one can also use a more delicate (1 + $\varepsilon$)-approximate $k$-means \citep{kumar2004simple} for a good approximate solution within a constant fraction of the optimal value. 
\end{myRmk}

\subsection{Discussion}
The theoretical challenges of our method lie in the following aspects. First, our assumptions regarding the rank-deficient connectivity matrices are considerably weaker than \citet{LeiJ2022Bias}, which further brings two barriers for establishing the Theorems  {\rm\ref{row_thm}} and {\rm\ref{column_thm}}. The first barrier is how the eigenvectors of the population SoG matrix relate to the true clusters. For this purpose, we systematically study the eigen-structure of population SoG matrix in Lemma \ref{row interpretation} (resp. Lemma \ref{column interpretation}) and provide interpretable conditions in Lemma \ref{row separable condition} (resp. Lemma \ref{column separable condition}) under which the eigenvectors of the population SoG matrices would reveal the underlying communities in multi-layer ScBMs. The second barrier lies in how to establish the lower bound of the smallest non-negative eigenvalues of the population SoG matrix. For this purpose, we carefully analyze the lower bound in Lemma \ref{kth eigenvalue}, which plays a crucial role for establishing the misclassification error bound. Second, the elements of noise term $\sum_{l=1}^LX_lX_l^T - {\rm diag}(\sum_{l=1}^LX_lX_l^T)$ exhibit complicated dependence which is caused by the quadratic form and asymmetry. In order to obtain a sharp bound, we use the decoupling techniques \citep{de1995decoupling,LeiJ2022Bias} to derive the concentration inequalities of the sum of quadratic asymmetric matrices; see Theorem \ref{sparse}.

\section{Simulations}\label{simulation}
In this section, we evaluate the finite sample performance of the proposed algorithm \textsf{DSoG}. To this end, we perform three experiments. The first one corresponds to the case where $\sum_{l=1}^LB_lB_l^T$ is of full rank, while the second one corresponds to the rank-deficient case. The third one is designed to mimic typical directed network structures.

\paragraph{\textbf{Methods for comparison.}} We compare our method \textsf{DSoG} with the following three methods.
\begin{itemize}
	\item \textsf{Sum}: spectral co-clustering based on the \emph{Sum} of adjacency matrices without squaring, that is, taking the left and right singular vectors of $\sum_{l=1}^L A_l$ as input of $k$-means clustering to obtain row and column clusters, respectively. 
	\item \textsf{SoG}: spectral co-clustering based on the \emph{Sum {o}f {G}ram} matrices, that is, taking the eigenvectors of the non-debiased matrices $\sum_{l=1}^L A_lA_l^T$ and $\sum_{l=1}^L A_l^TA_l$ as the input of $k$-means clustering to obtain row and column clusters, respectively. 
	\item \textsf{MASE}: the method called \emph{Multiple Adjacency Spectral Embedding} \citep{arroyo2021inference}, where to obtain row and column clusters, the eigenvectors of $\sum_{l=1}^L U_l U_l^T/L$ and $\sum_{l=1}^L V_l V_l^T/L$ are used as the input of the $k$-means clustering with $U_l$ and $V_l$ representing the singular vectors of the layer-wise matrices $A_l$.
\end{itemize}

\paragraph{\textbf{Experiment 1.}} The networks are generated from the multi-layer ScBM via the mechanism given in \eqref{adj_generate}. We consider $n = 500$ nodes per network across $K_y = 3$ row clusters and $K_z = 3$ column clusters, with row cluster sizes $n_1^y = 200, \ n_2^y = 100,\  n_3^y = 200$ and column cluster sizes $n_1^z = 150, \ n_2^z = 200,\  n_3^z = 150$. We fix $L = 50$ and set $B_l = \rho B^{(1)}$ for $l \in \{1,\ldots, L/2\}$, and $B_l = \rho B^{(2)}$ for $l \in \{L/2+1,\ldots, L\}$, with 
	\begin{equation*}
		B^{(1)} = U\begin{bmatrix} 1.5 & 0 & 0 \\ 0 & 0.2 & 0 \\ 0 & 0 & 0.4 \end{bmatrix}V^T \approx \begin{bmatrix}0.46 & 0.625 & 0.225 \\ 0.46 & 0.225 & 0.625 \\ 0.85 & 0.46 & 0.46 \end{bmatrix}
	\end{equation*} 
	and
	\begin{equation*}
		B^{(2)} = U\begin{bmatrix} 1.5 & 0 & 0 \\ 0 & 0.2 & 0 \\ 0 & 0 & -0.4 \end{bmatrix}V^T \approx \begin{bmatrix}0.46 & 0.225 & 0.625  \\ 0.46 & 0.625 & 0.225 \\ 0.85 & 0.46 & 0.46 \end{bmatrix},
	\end{equation*} 
	where 
	\begin{equation*}
		U = \begin{bmatrix}
			1/2 & 1/2 & -\sqrt{2}/2 \\ 1/2 & 1/2 & \sqrt{2}/2 \\ \sqrt{2}/2 & -\sqrt{2}/2 & 0
		\end{bmatrix} \quad \mbox{and} \quad
		V = \begin{bmatrix}
			\sqrt{2}/2 & -\sqrt{2}/2 & 0 \\ 1/2 & 1/2 & -\sqrt{2}/2 \\ 1/2 & 1/2 & \sqrt{2}/2
		\end{bmatrix}.		
	\end{equation*}
	To measure the effect of network sparsity, we vary the overall edge density parameter $\rho$ in the range of 0.03 to 0.16. It is obvious that direct summation of $B_l$ would lead to the confusion of first and second clusters. Hence, it is expected that \textsf{Sum} would not perform well under this case.

\paragraph{\textbf{Experiment 2.}} In this experiment, we consider the case where $\sum_{l=1}^LB_lB_l^T$ (or $\sum_{l=1}^LB_lB_l^T$) is rank-deficient. Specifically, we consider the following model. We fix $L = 50$ and set $B_l = \rho B^{(1)}$ for $l \in \{1,\ldots, L/2\}$, and $B_l = \rho B^{(2)}$ for $l \in \{L/2+1,\ldots, L\}$, with
	\begin{equation*}
		B^{(1)} = U\begin{bmatrix} 1.2 & 0 & 0 \\ 0 & 0.4 & 0 \\ 0 & 0 & 0 \end{bmatrix}V^T = \begin{bmatrix} 0.59 & 0.14 & 0.33 \\ 0.14 & 0.39 & 0.47 \\ 0.33 & 0.47 & 0.63 \end{bmatrix}
	\end{equation*} 
	and
	\begin{equation*}
		B^{(2)} = U\begin{bmatrix} 1.2 & 0 & 0 \\ 0 & -0.4 & 0 \\ 0 & 0 & 0 \end{bmatrix}V^T = \begin{bmatrix} 0.01 & 0.46 & 0.52  \\ 0.46 & 0.21 & 0.37 \\ 0.52 & 0.37 & 0.57 \end{bmatrix},
	\end{equation*} 
	where 
	\begin{equation*}
		U = V \approx \begin{bmatrix}
			0.5 & 0.84 & -0.19 \\ 0.5 & -0.46 & -0.73 \\ 0.71 & -0.27 & 0.65
		\end{bmatrix}
	\end{equation*}
	and $\rho$ is the overall edge density parameter which varies from 0.03 to 0.16. It is easy to see that $\sum_{l=1}^LB_lB_l^T$ is rank-deficient  and the rank is 2. As in Experiment 1, we consider $n = 500$ nodes per network across $K_y = 3$ row clusters and $K_z = 3$ column clusters, with row cluster sizes $n_1^y = 100, \ n_2^y = 150,\  n_3^y = 250$ and column cluster sizes $n_1^z = 100, \ n_2^z = 250,\  n_3^z = 150$. With this set-up, we generate the adjacency matrix with respect to each layer via (\ref{adj_generate}).
		
\paragraph{\textbf{Experiment 3.}} 
In this experiment, we generate directed networks with `transmission' nodes or `message passing' nodes. In networks with `transmission' nodes, there exists a set of nodes that only receive edges from one set of nodes and send edges to another set of nodes, and hence are called ‘transmission nodes’. Hence, the sending clusters and receiving clusters are distinct. See Figure \ref{topological} (a) and (b) for illustration. In networks with `message passing' nodes, the edges spanning different communities start in upper communities and extend down to lower communities, just like passing messages. See Figure \ref{topological}(c) for illustration, where the row clusters and column clusters are identical. In our set-up, we fix $L = 30$ and set $B_l = \rho B^{(1)}$ for $l \in \{1, \ldots, L/3\}$,  $B_l = \rho B^{(2)}$ for $l \in \{L/3+1, \ldots, 2L/3\}$, and $B_l = \rho B^{(3)}$ for $l \in \{2L/3+1, \ldots, L\}$, with 
	\begin{equation*}
		B^{(1)} = \begin{bmatrix}
			0.3 & 0 & 0 \\ 0 & 0.2 & 0 \\ 0 & 0 & 0.3
		\end{bmatrix}, \quad 
		B^{(2)} = \begin{bmatrix}
			0 & 0 & 0 \\ 0.3 & 0 & 0 \\ 0.5 & 0.3 & 0
		\end{bmatrix}, \quad 
  		B^{(3)} = \begin{bmatrix}
			0 & 0.2 & 0.2 \\ 0 & 0 & 0.2 \\ 0 & 0 & 0
		\end{bmatrix}.  \quad 
	\end{equation*}
Under $B^{(1)}$, we let the row and column clusters to be different with transmission nodes; see Figure \ref{topological} (a) and (b) for illustration. Under $B^{(2)}$ and $B^{(3)}$, we incorporate the message passing nodes and the row clusters and column clusters are identical to those corresponds to $B^{(1)}$; see Figure \ref{topological} (c) for illustration. Specifically, we consider $n = 300$ nodes per network across $K_y = 3$ row clusters and $K_z = 3$ column clusters, with row cluster sizes $n_1^y = 120, \ n_2^y = 100,\  n_3^y = 80$ and column cluster sizes $n_1^z = 80, \ n_2^z = 100,\  n_3^z = 120$. With this set-up, we generate the adjacency matrix for each layer via (\ref{adj_generate}), with the overall edge density parameter $\rho$ varying from 0.03 to 0.16.
\begin{figure*}[!htbp]
    \centering
    \begin{subfigure}{0.3\textwidth}\includegraphics[width=\textwidth]{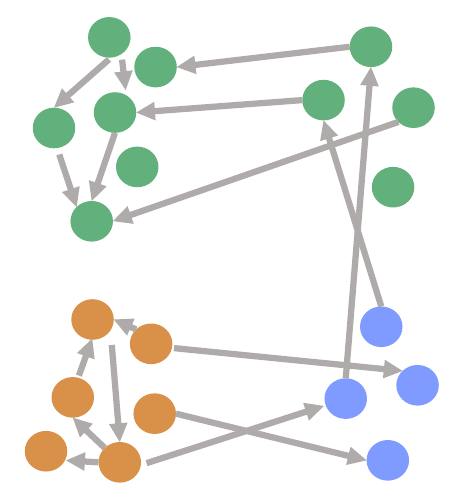}
        \caption{}
    \end{subfigure}
    \hspace{0.01\textwidth}
    \begin{subfigure}{0.3\textwidth}
        \includegraphics[width=\textwidth]{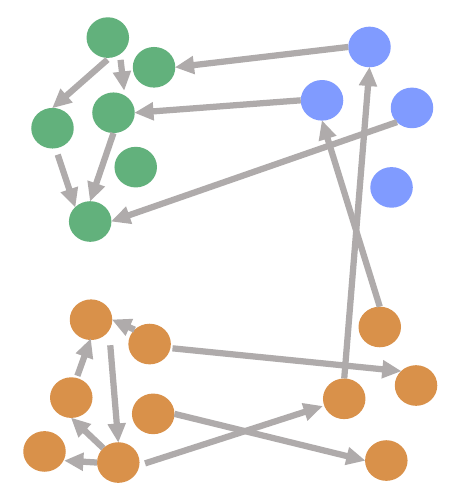}
        \caption{}
    \end{subfigure}
    \hspace{0.01\textwidth}
    \begin{subfigure}{0.3\textwidth}
        \includegraphics[width=\textwidth]{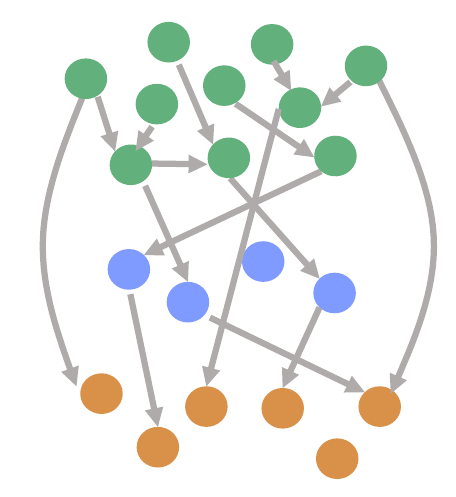}
        \caption{}
    \end{subfigure}
    \caption{Illustration for networks under different block  probability matrices. Colors indicate communities. (a) row clusters under $B^{(1)}$; (b) column clusters under $B^{(1)}$; (c) Row clusters under $B^{(2)}$ ($B^{(3)}$); column clusters are analogous.}
    \label{topological}
\end{figure*}

\begin{figure*}[!htbp]
    \centering
    \begin{subfigure}[b]{0.48\textwidth}
        \includegraphics[width=\textwidth]{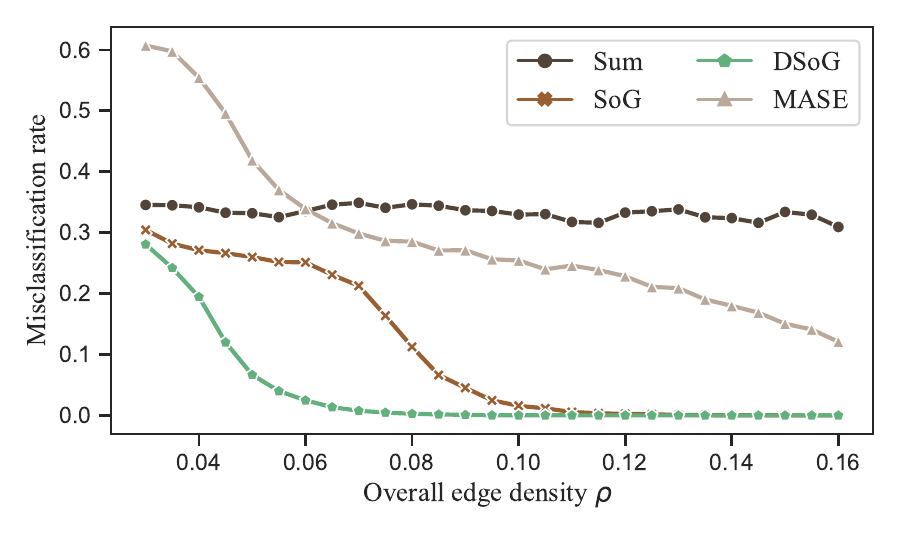}
        \vspace{-0.13\textwidth}
        \caption{Experiment 1, row clustering}
    \end{subfigure}\hspace{0.3cm}
    \begin{subfigure}[b]{0.48\textwidth}
        \includegraphics[width=\textwidth]{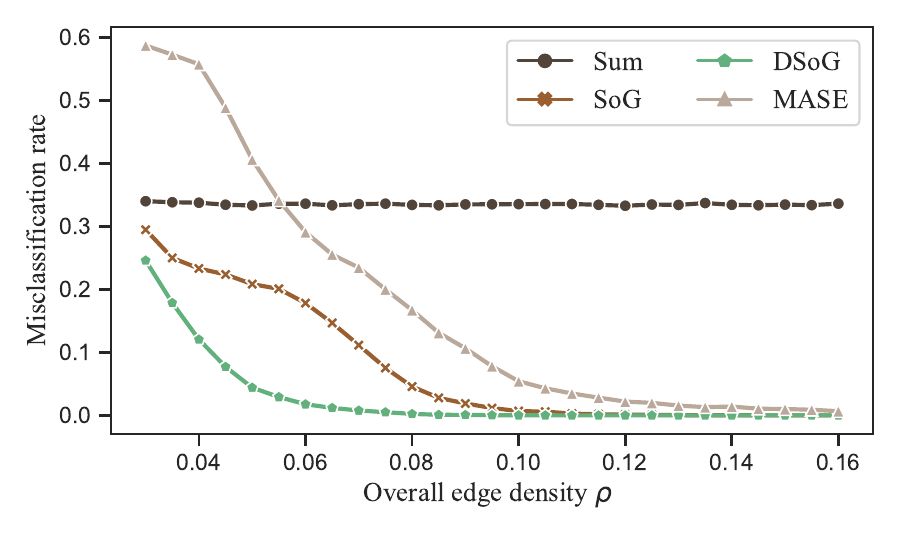}
        \vspace{-0.13\textwidth}
        \caption{Experiment 1, column clustering}
    \end{subfigure}\\
    \vspace{0.035\textwidth}
    \begin{subfigure}[b]{0.48\textwidth}
        \includegraphics[width=\textwidth]{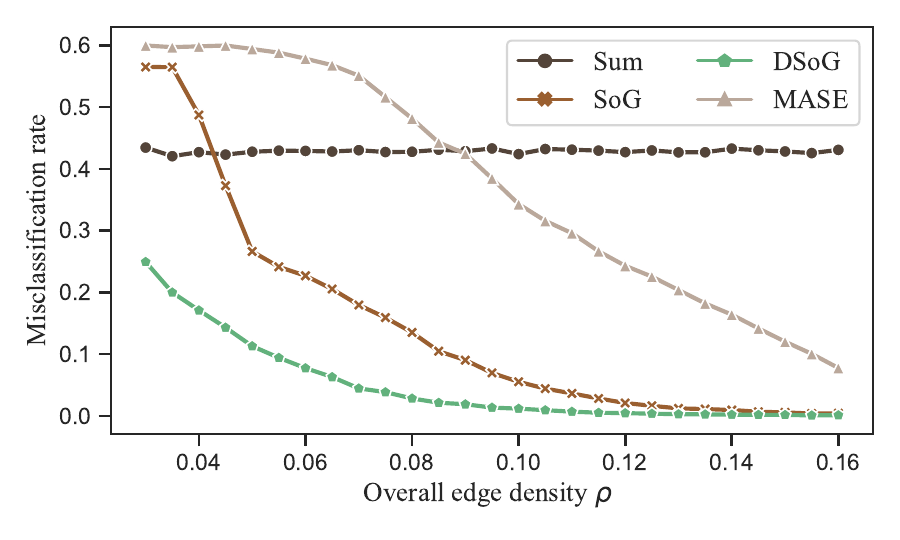}       
        \vspace{-0.13\textwidth}
        \caption{Experiment 2, row clustering}
    \end{subfigure}
    \begin{subfigure}[b]{0.48\textwidth}
        \includegraphics[width=\textwidth]{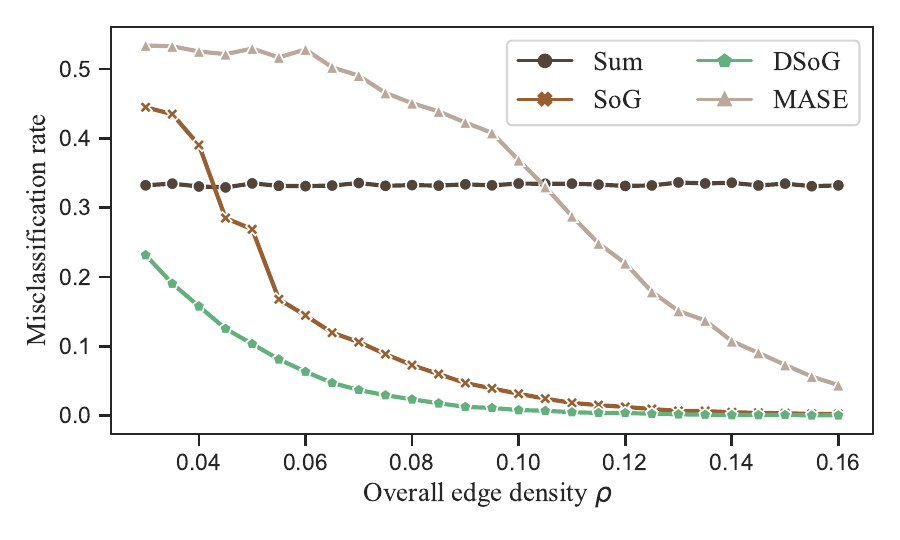}       
        \vspace{-0.13\textwidth}
        \caption{Experiment 2, column clustering}
    \end{subfigure}\\
    \vspace{0.035\textwidth}
    \begin{subfigure}[b]{0.48\textwidth}
        \includegraphics[width=\textwidth]{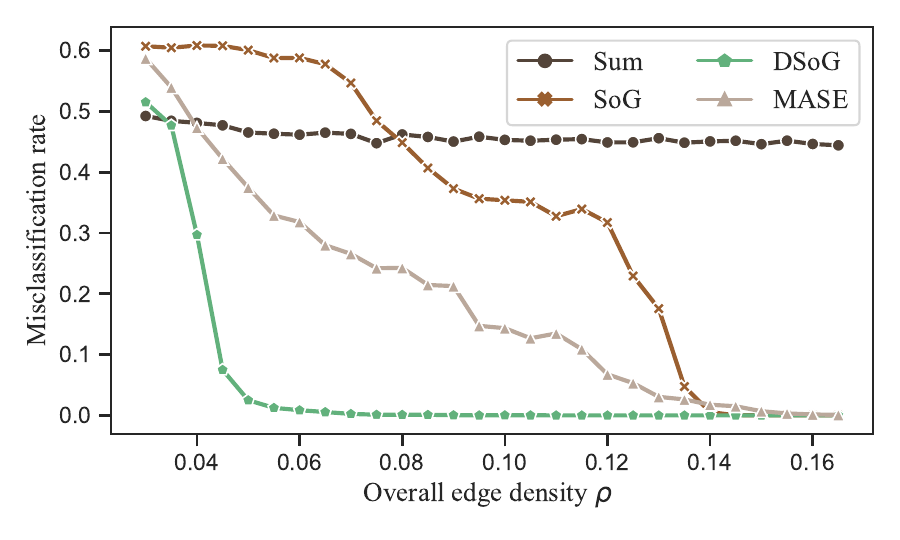}       
        \vspace{-0.13\textwidth}
        \caption{Experiment 3, row clustering}
    \end{subfigure}
    \begin{subfigure}[b]{0.48\textwidth}
        \includegraphics[width=\textwidth]{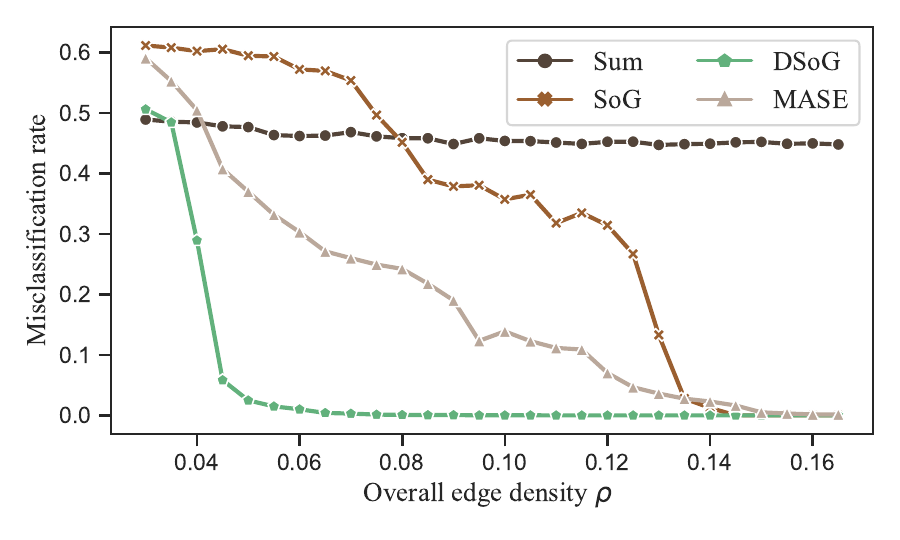}       
        \vspace{-0.13\textwidth}
        \caption{Experiment 3, column clustering}
    \end{subfigure} 
    \caption{Misclassification rates of four methods with varying $\rho$ in three experimental set-ups. }
    \label{simulation result}
\end{figure*}

\paragraph{\textbf{Results.}} We use the misclassification rate defined in (\ref{misclustering rate}) to measure the proportion of misclassified (up to permutations) nodes. The embedding dimensions and the number of clusters are all set to their true values. The averaged results over 50 replications for Experiments 1-3 are displayed in Figure \ref{simulation result}, where we vary $\rho$ from 0.03 to 0.16 in 27 equally-spaced values.  The results demonstrate that our proposed method \textsf{DSoG} has a noticeable impact on the accuracy of clustering. Specifically, the \textsf{Sum} method performs poorly due to the fact that some eigen-components cancel out in the summation. The \textsf{DSoG} method has a significant advantage over the \textsf{SoG} method, especially when $\rho$ is small, i.e., the very sparse regime, which is consistent with our theoretical results. We can also observe that our method \textsf{DSoG} outperforms the \textsf{MASE} method in all settings. 

It is important to emphasize that in the above experiments, the community cancellation leads to the inferior performance of \textsf{Sum}, while we note that when the direct summation neither causes community cancellation nor signal reduction, \textsf{Sum} is comparable or slightly better than \textsf{DSoG}. Since we never know the truth in real applications, \textsf{DSoG} provides a safe and satisfactory estimator for co-clustering multi-layer directed networks.

In Appendix \ref{AppendixE}, we also conduct an experiment to study the sensitivity of tuning parameters for each method. It turns out that \textsf{DSoG} continues to outperform other methods, although there exists certain performance degradation for all methods.

In the above experiments, we focused on the comparison of \textsf{DSoG} with spectral clustering-based methods. In Appendix \ref{AppendixE}, we also compare \textsf{DSoG} with likelihood-based methods by \cite{wang2021fast} and \cite{fu2023profile}. It turns out that these likelihood-based methods show unstable and inferior performance under the random initialization or initialized by \textsf{Sum}. While when choosing \textsf{DSoG} as the initial estimator, the likelihood-based method \textsf{ML-PPL} shows better performance than our method \textsf{DSoG}. Hence, the proposed method provides a good initial estimator for these likelihood-based methods when there exists community cancellation via direct summation.
In addition, we scale the number of nodes to 1000 and compare the running time of each method. It turns out that \textsf{DSoG} is more computationally efficient than likelihood-based methods.
See Appendix \ref{AppendixE} for the detailed experiments and results.

\section{Real data analysis}\label{realdata}
In this section, we analyze the WFAT dataset, which is a public dataset collected by the Food and Agriculture Organization of the United Nations. The original data includes trading records for more than 400 food and agricultural products imported/exported annually by all the countries in the world. The dataset is available at \href{https://www.fao.org}{https://www.fao.org}. As described in other works analyzing these data \citep{de2015structural,jing2021community,noroozi2022sparse}, it can be considered as a multi-layer network, where layers represent food and agricultural products, and nodes are countries and edges at each layer represent import/export relationships of a specific food and agricultural product among countries. Specifically, \cite{noroozi2022sparse} treated this dataset as a \emph{mixture} multi-layer SBMs \citep{pensky2021clustering}, that is, the network layers are classified to different groups, the within group layers share common community structures. The method identifies three groups of layers. Therefore, to fit into our assumption that the communities are consensus among layers, we select a group identified by \cite{noroozi2022sparse}. This group including 24 products contains mostly cereals, stimulant crops and derived products, see Table \ref{items} for details.


\begin{table}[!htbp]
\centering
\begin{tabular}[]{|c|}
\hline
``Pastry", ``Rice, paddy", ``Rice, milled", ``Breakfast cereals", Mixes and doughs", \\``Food preparations of flour, meal or malt extract", ``Sugar and syrups n.e.c.", \\``Sugar confectionery",  ``Communion wafers and similar products.", ``Prepared nuts", \\ ``Vegetables preserved (frozen)", ``Juice of fruits n.e.c.", ``Fruit prepared n.e.c.",\\ ``Orange juice", ``Other non-alcoholic caloric beverages", ``Food wastes", \\ ``Other spirituous beverages", ``Coffee, decaffeinated or roasted", ``Coffee, green", \\``Chocolate products nes", ``Pepper, raw", ``Dog or cat food, put up for retail sale", \\``Food preparations n.e.c.", ``Crude organic material n.e.c."\\
\hline
\end{tabular}
    \caption{List of a group of food and agricultural products we considered, which consists of 24 different products involving cereals, stimulant crops, and derived products.}
    \label{items}
\end{table}

\paragraph{\textbf{Data preprocessing.}} We convert the original trading data into directed networks, which is distinct from all previously described efforts to analyze this data, where the trading data is reduced to an undirected network. We focus on the trading data in the year 2020. To create a directed trading network for each of the 24 products, we draw a directed edge from the exporter to the importer if the export/import value of the product exceeds \$10000. This particular threshold would yield sparse networks that have many disjoint connected components individually but have one connected component after aggregation. It is important to mention that the performance of \textsf{DSoG} is insensitive to the specific product value threshold.
We then remove all the countries (nodes) whose total in-degree or out-degree across all 24 layers is less than 14. We choose this value to make sure that each node has links to at least one of the other nodes in at least half of the layers. Indeed, the average total in-degree and out-degree of nodes which do not have any neighbors in 12 or more of the layers (i.e., more than half of the layers have a zero out-degree or in-degree) is 13. As a result, we obtain a multi-layer network with 24 layers and 142 nodes per layer. Subsequently, we reconstruct the row and column clusters using the proposed algorithm. 

\paragraph{\textbf{Tuning parameters selection.}} There are four tuning parameters $K_y, K_z, K$ and $K'$ in real data analysis. 
To select the embedding dimensions $K$ and $K'$, we use the scree plot method, which is widely used in network data analysis. The scree plot of the top eigenvalues of the debiased SoG matrices over the layer-wise adjacency matrices are shown in Figure \ref{scree plot}. For both of the row and column clustering, there is an elbow on the scree plot at the 4th position. Hence, we choose $K = 4$ and $K' = 4$ in the WFAT data analysis.
\begin{figure*}[h]
    \centering
    \begin{subfigure}{0.35\textwidth}\includegraphics[width=\textwidth]{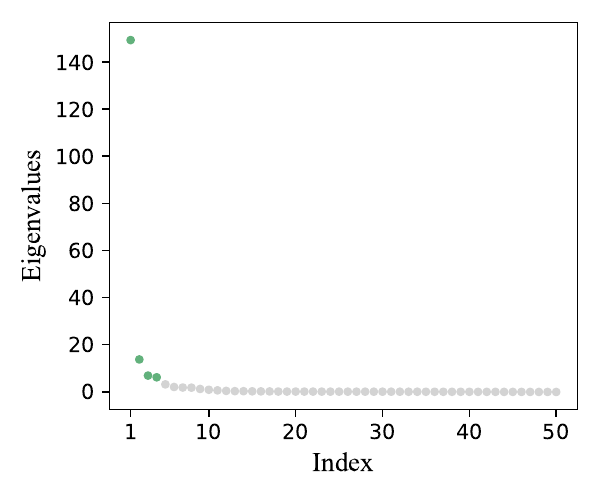}
        \caption{row clustering}
    \end{subfigure}
    \hspace{0.1\textwidth}
    \begin{subfigure}{0.35\textwidth}
        \includegraphics[width=\textwidth]{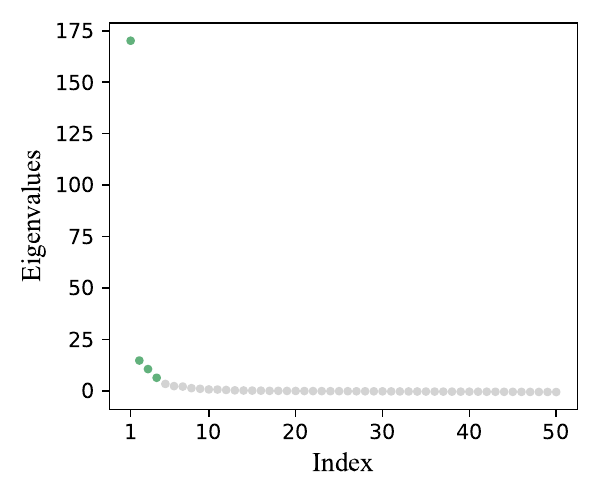}
        \caption{column clustering}
    \end{subfigure}
    \caption{Scree plots of the top eigenvalues of DSoG matrices with respect to row clustering and column clustering for the WFAT data.}
    \label{scree plot}
\end{figure*}

\begin{figure*}[!h]
    \centering
    \begin{subfigure}{0.35\textwidth}\includegraphics[width=\textwidth]{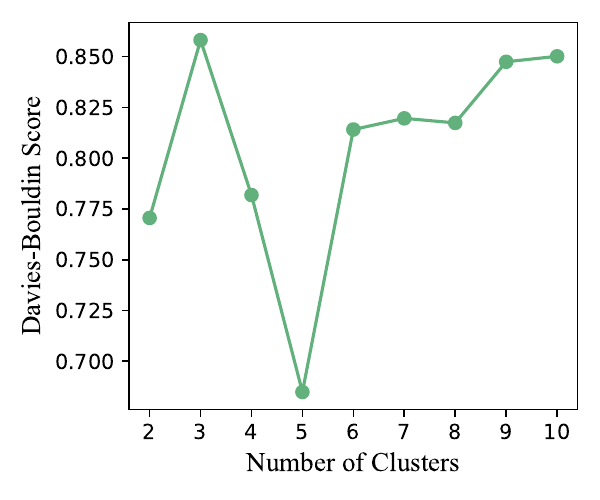}
        \caption{row clustering}
    \end{subfigure}
    \hspace{0.1\textwidth}
    \begin{subfigure}{0.35\textwidth}
        \includegraphics[width=\textwidth]{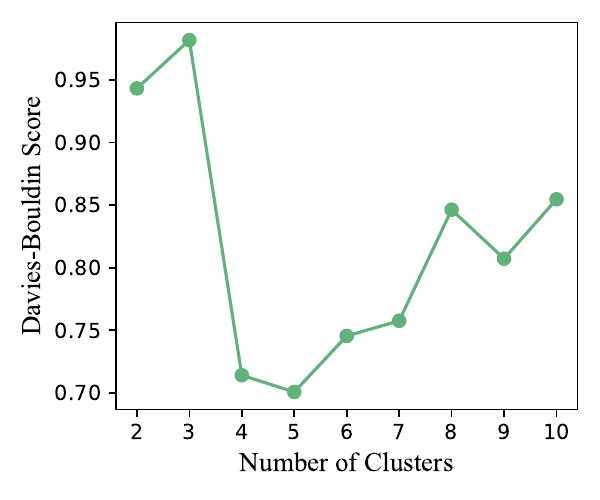}
        \caption{column clustering}
    \end{subfigure}
    \caption{Davies-Bouldin scores under different clustering numbers with respect to row clustering and column clustering for the WFAT data.
    }
    \label{dbscore}
\end{figure*}

To select the number of clusters $K_y$ and $K_z$, we use the model-free method Davies-Bouldin score \citep{davies1979cluster}, because current literature for selecting the number of clusters violated our rank-deficient assumption \citep{li2020network,hu2020corrected}. 
Figure \ref{dbscore} shows the Davies-Bouldin scores for the row and column clustering, where a smaller value indicates a more compact and well-separated cluster. The Davies-Bouldin scores turn out to attain their minimum at $5$ for both the row and column clustering. Hence, we set $K_y=K_z= 5$ in the WFAT data analysis.

\paragraph{\textbf{Results.}}
The estimated row and column clusters are displayed in Figure \ref{map}. The clusters of countries are approximately related to their geographic locations, which is coherent with the economic laws of world trade. Specifically, for the row clusters (see Figure \ref{map}(a)), Community 1 mainly includes countries in Africa, West and Central Asia; Community 2 is composed of countries in Central and South America; Community 3 includes China, the United States, Western and Southern Europe countries; Community 4 involves Southeast Asia and Oceania countries; Community 5 consists of the remaining European countries. For the column clusters  (see Figure \ref{map}(b)), Community 1 mainly includes Africa and West Asian countries; Community 2 consists of countries in Central and South America; Community 3 includes Oceania, East and South Asian countries; Community 4 involves Eastern European and some West Asian countries; Community 5 mainly includes Western Europe, North America and Russia. 
\begin{figure*}[!htbp]
    \centering
    \begin{subfigure}[b]{0.95\textwidth}
        \includegraphics[width=\textwidth]{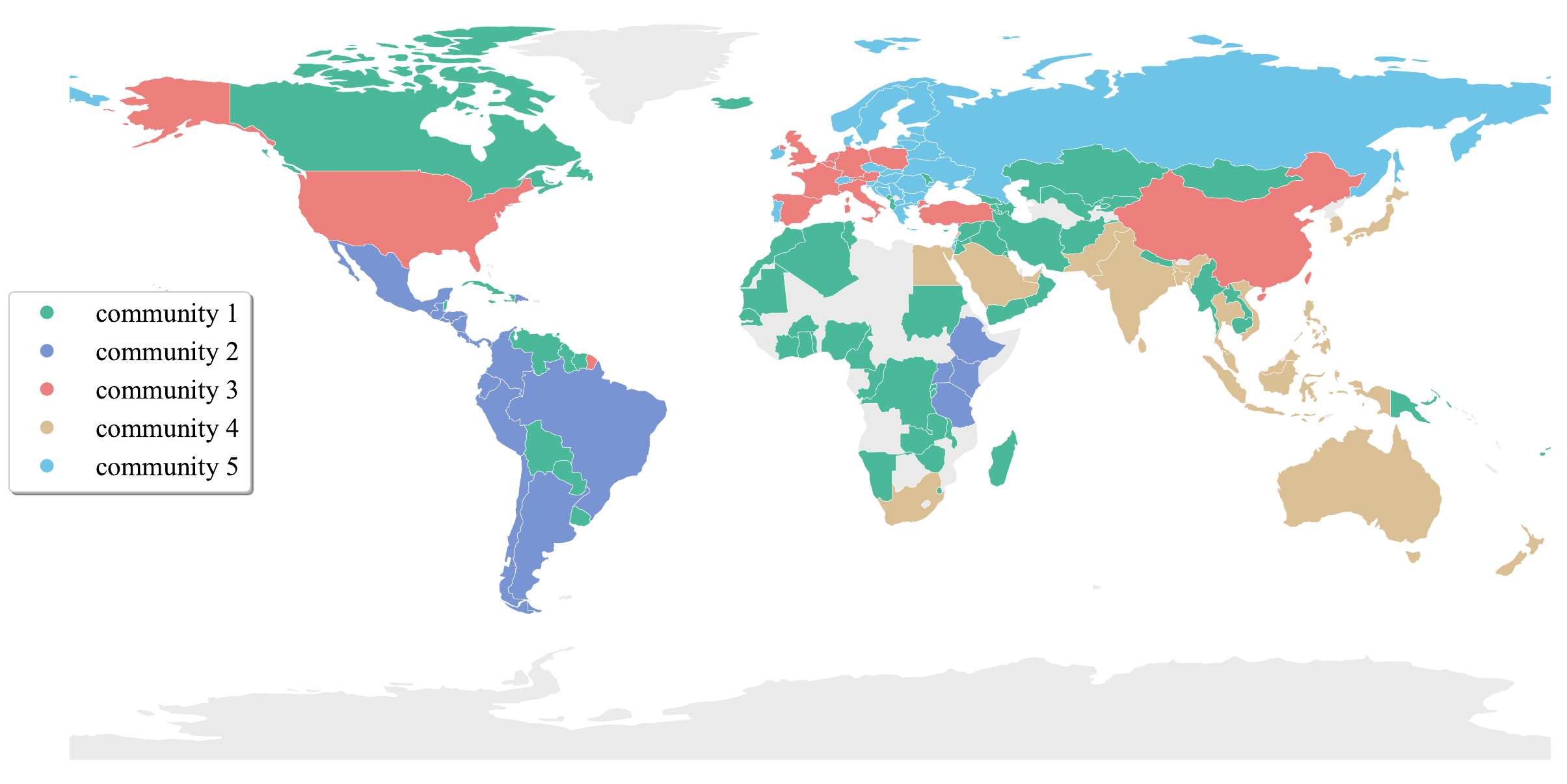}
        \caption{row (export) clustering}
    \end{subfigure}\\
    \vspace{0.02\textwidth}
    \begin{subfigure}[b]{0.95\textwidth}
        \includegraphics[width=\textwidth]{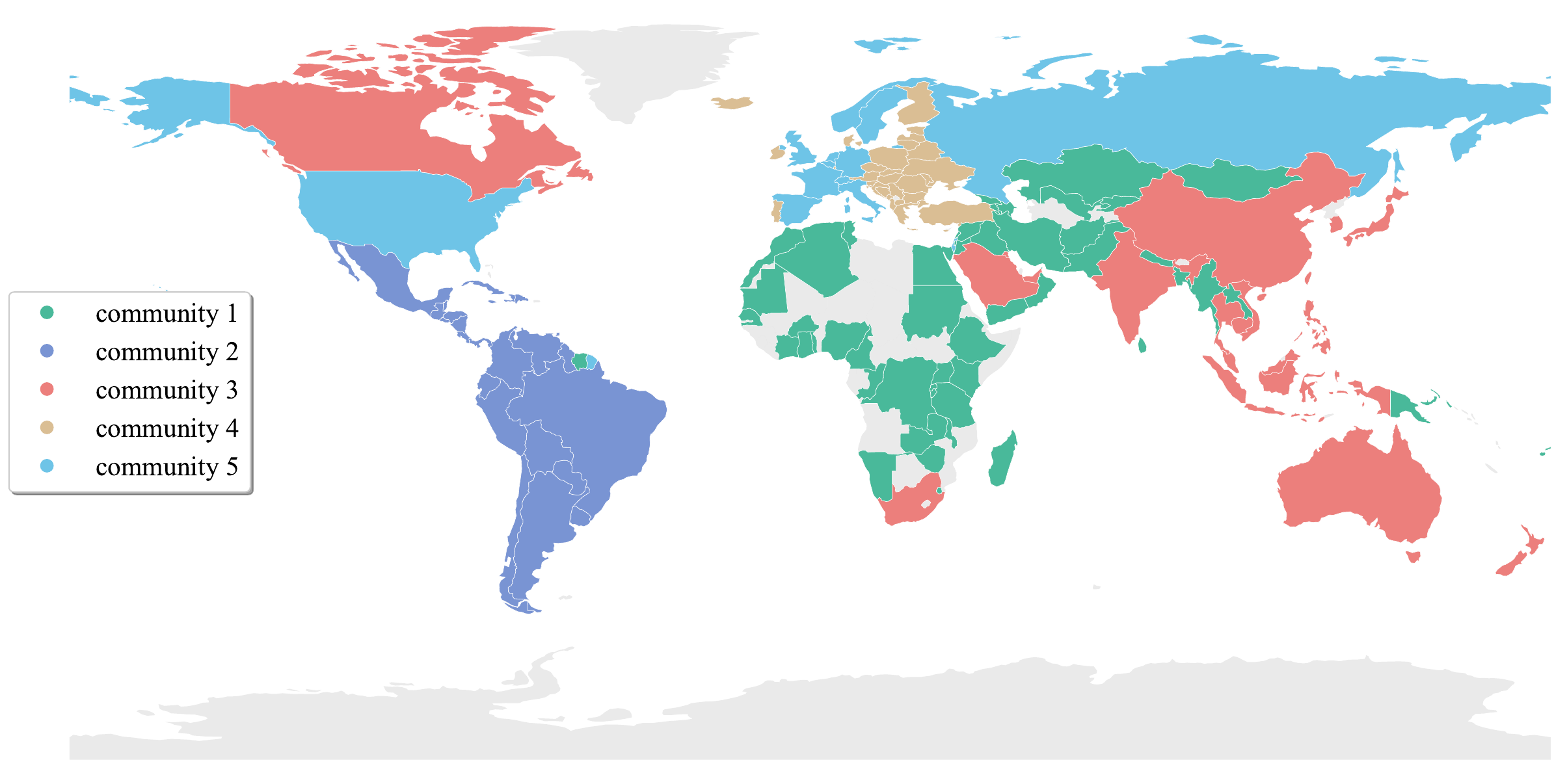}
        \caption{column (import) clustering}
    \end{subfigure}
    \caption{Community structure separation of food trade networks (directed) containing 142 major countries. Colors indicate communities, 
    where light gray corresponds to countries that do not participate in clustering.}
    \label{map}
\end{figure*}
 
\begin{figure*}[!htbp]{}
    \centering
    \includegraphics[width=0.95\textwidth]{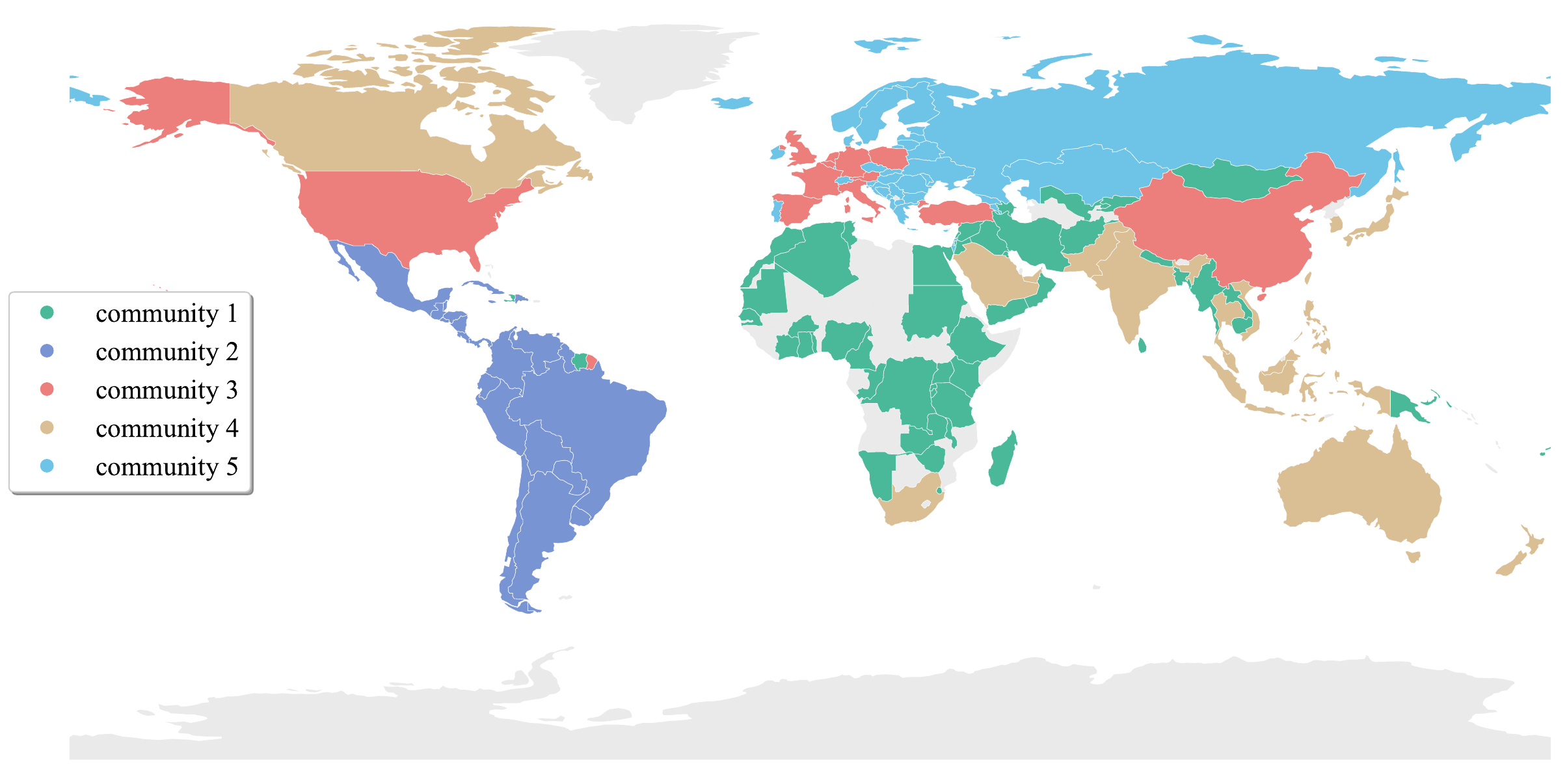}
    \caption{Community structure separation of food trade networks (undirected) containing 142 major countries. Colors indicate communities, 
    where light gray corresponds to countries that do not participate in clustering.}
    \label{map_undirected}
\end{figure*}

We can see that the structure of the row (export) clusters is not identical to that of the column (import) clusters, which is insightful and more realistic than what would be expected from undirected multi-layer networks. For instance, China is grouped with major European economies and America in the row clusters partition, while it is grouped mainly with East and South Asia countries in the column clusters partition, indicating that for the products in Table \ref{items}, China is aligned with major world economies in its export trade, while it is aligned mainly with neighboring countries in its import trade. This observation is entirely plausible. Export patterns primarily reflect a nation's productive capacity and dominant industries. Major global economies with robust food processing industries and agricultural technology tend to export value-added manufactured goods. Conversely, import patterns primarily mirror a country's consumer demand. Due to similar geographical conditions, population sizes and food consumption habits, countries in close proximity to each other often exhibit similar import patterns. 

By contrast, we compare \textsf{DSoG} with its undirected analog, \textsf{DSoS}. For \textsf{DSoS}, we ignore the direction of edges to obtain symmetrized adjacency matrices $A'_l$, and then perform eigendecomposition and $k$-means on the  debiased $\sum_{l=1}^L(A'_l)^2$ \citep{LeiJ2022Bias}. The results of DSoS are shown in Figure \ref{map_undirected}. We use the Adjusted Rand Index (ARI) to measure the similarity between \textsf{DSoG} and \textsf{DSoS} in terms of row clustering and column clustering, where the single set of clusters obtained by \textsf{DSoS} is regarded as both the row and column clusters, and a larger ARI indicates more similarities. It turns out that for the column (import) clusters, the ARI is 0.74, while for the row (export) clusters, the ARI is 0.53. Therefore, the results of \textsf{DSoS} are less aligned with the row clusters of \textsf{DSoG}, which leads to the failure of \textsf{DSoS} in recognizing the export patterns. 

\begin{figure*}[t]{}
    \centering
    \includegraphics[width=0.6\textwidth]{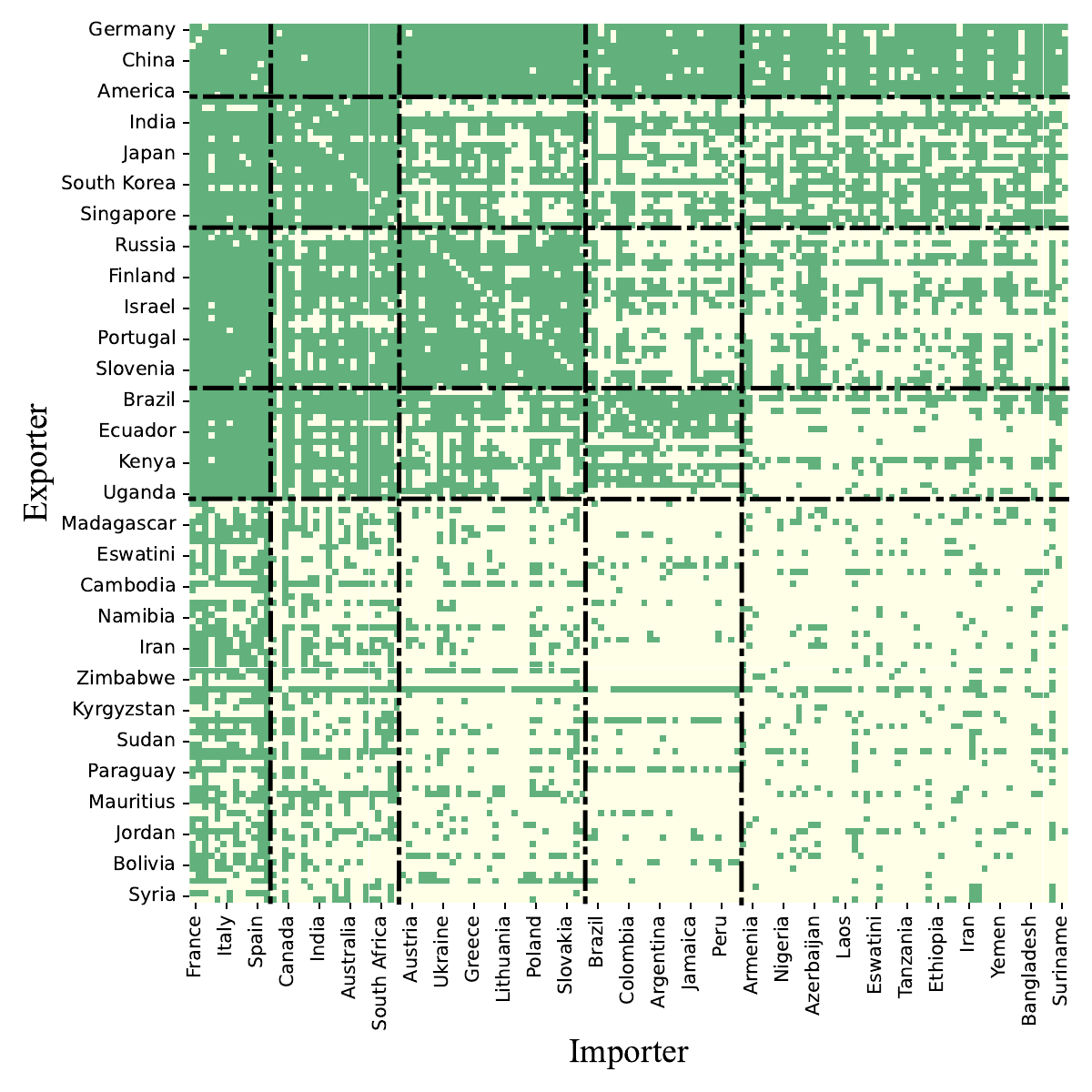}    
    \caption{Visualization of the aggregated adjacency matrix. Yellow pixels correspond to the absence of an edge between the corresponding countries while green pixels correspond to an edge. The black dashed lines in the figures signifies the division between clusters. }
    \label{Adj_sum}
\end{figure*}

\begin{figure*}[!htbp]
    \centering
    \begin{subfigure}[b]{0.37\textwidth}
        \includegraphics[width=\textwidth]{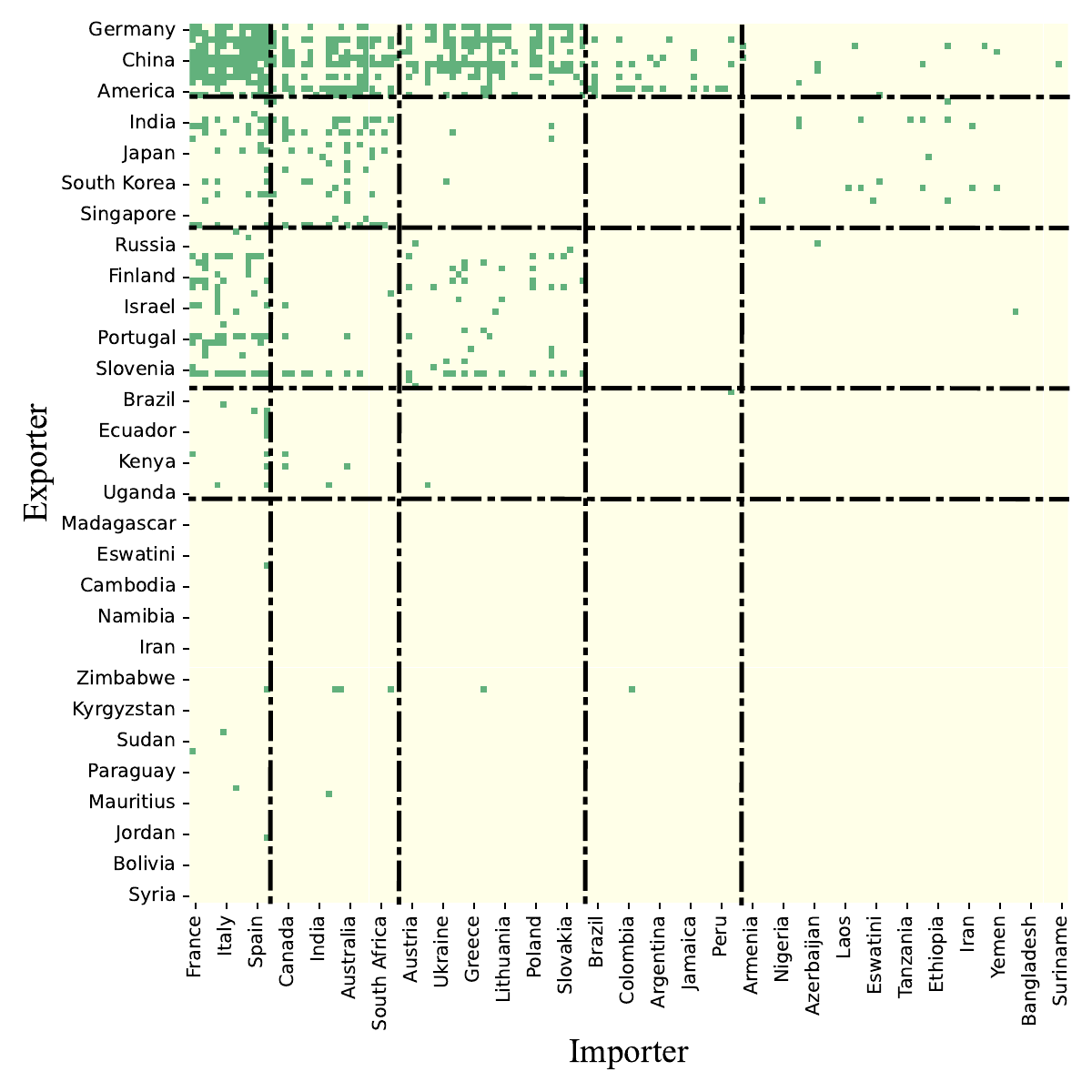}
        \vspace{-0.16\textwidth}
        \caption{Vegetables preserved (frozen)}
    \end{subfigure}
    \hspace{0.05\textwidth}
    \begin{subfigure}[b]{0.37\textwidth}
        \includegraphics[width=\textwidth]{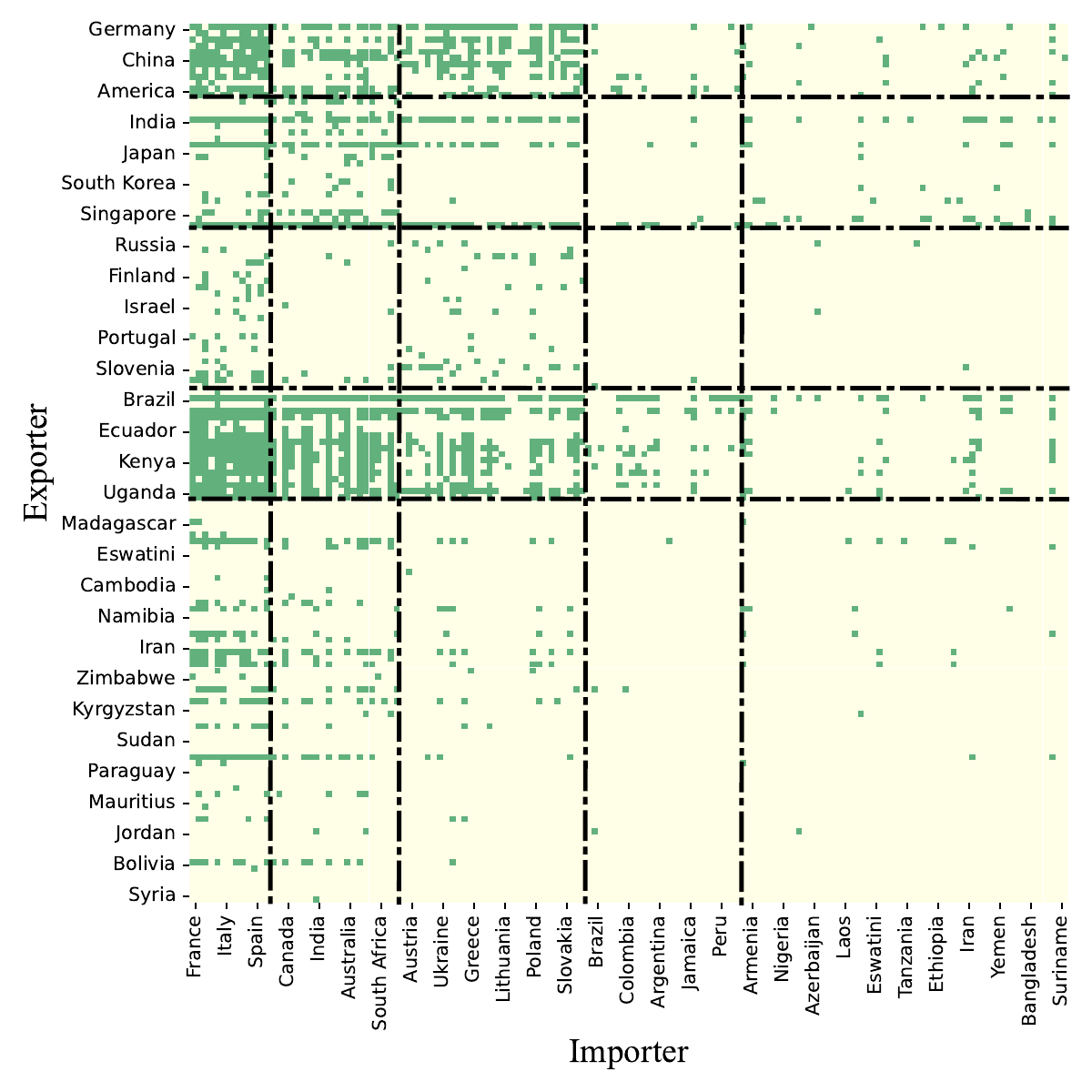}
        \vspace{-0.16\textwidth}
        \caption{Coffee, green}
    \end{subfigure}\\
    \vspace{0.03\textwidth}
    \begin{subfigure}[b]{0.37\textwidth}
        \includegraphics[width=\textwidth]{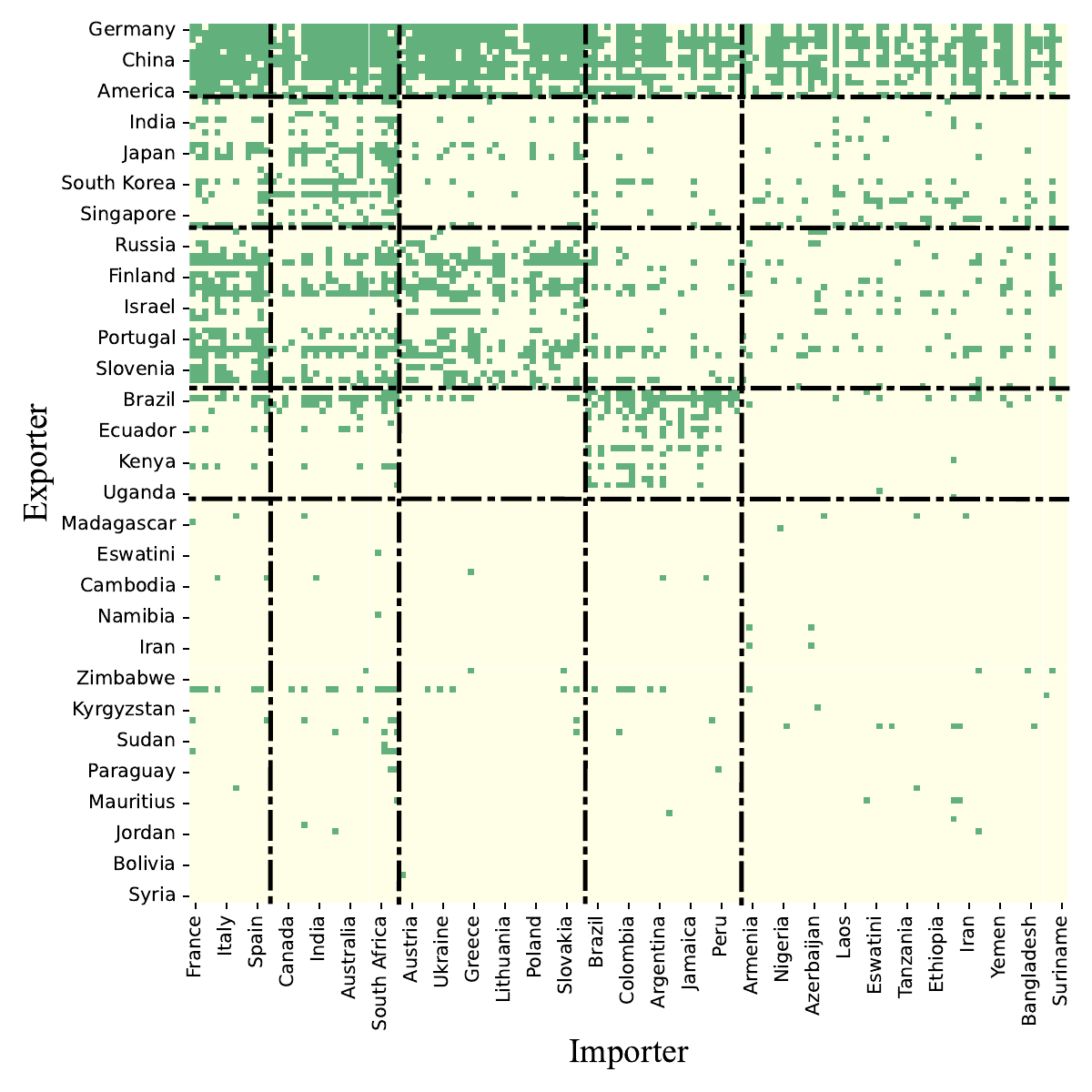}
        \vspace{-0.16\textwidth}
        \caption{Food wastes}
    \end{subfigure}
    \hspace{0.05\textwidth}
    \begin{subfigure}[b]{0.37\textwidth}
        \includegraphics[width=\textwidth]{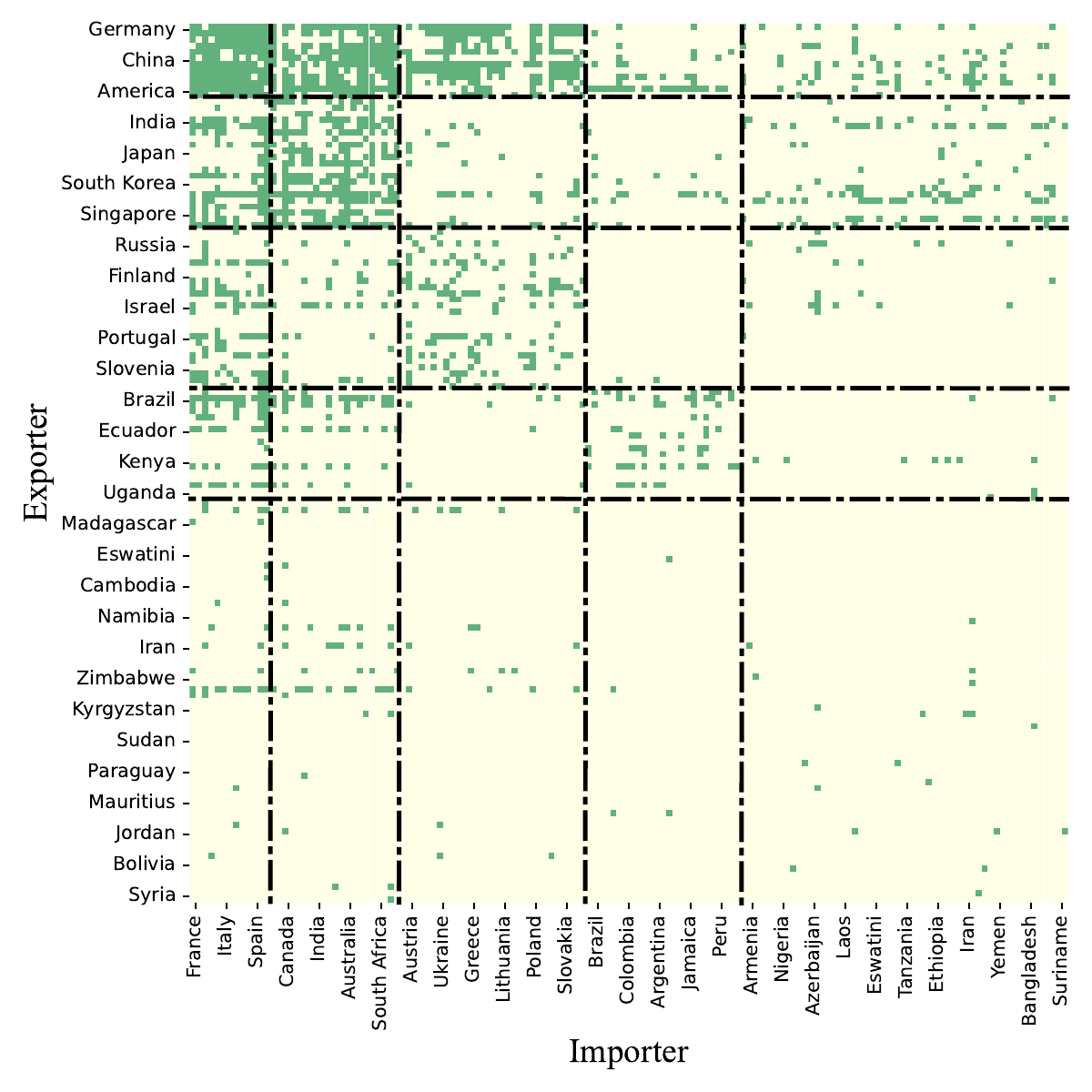}
        \vspace{-0.16\textwidth}
        \caption{Juice of fruits n.e.c.}
    \end{subfigure}\\
    \vspace{0.03\textwidth}
    \begin{subfigure}[b]{0.37\textwidth}
        \includegraphics[width=\textwidth]{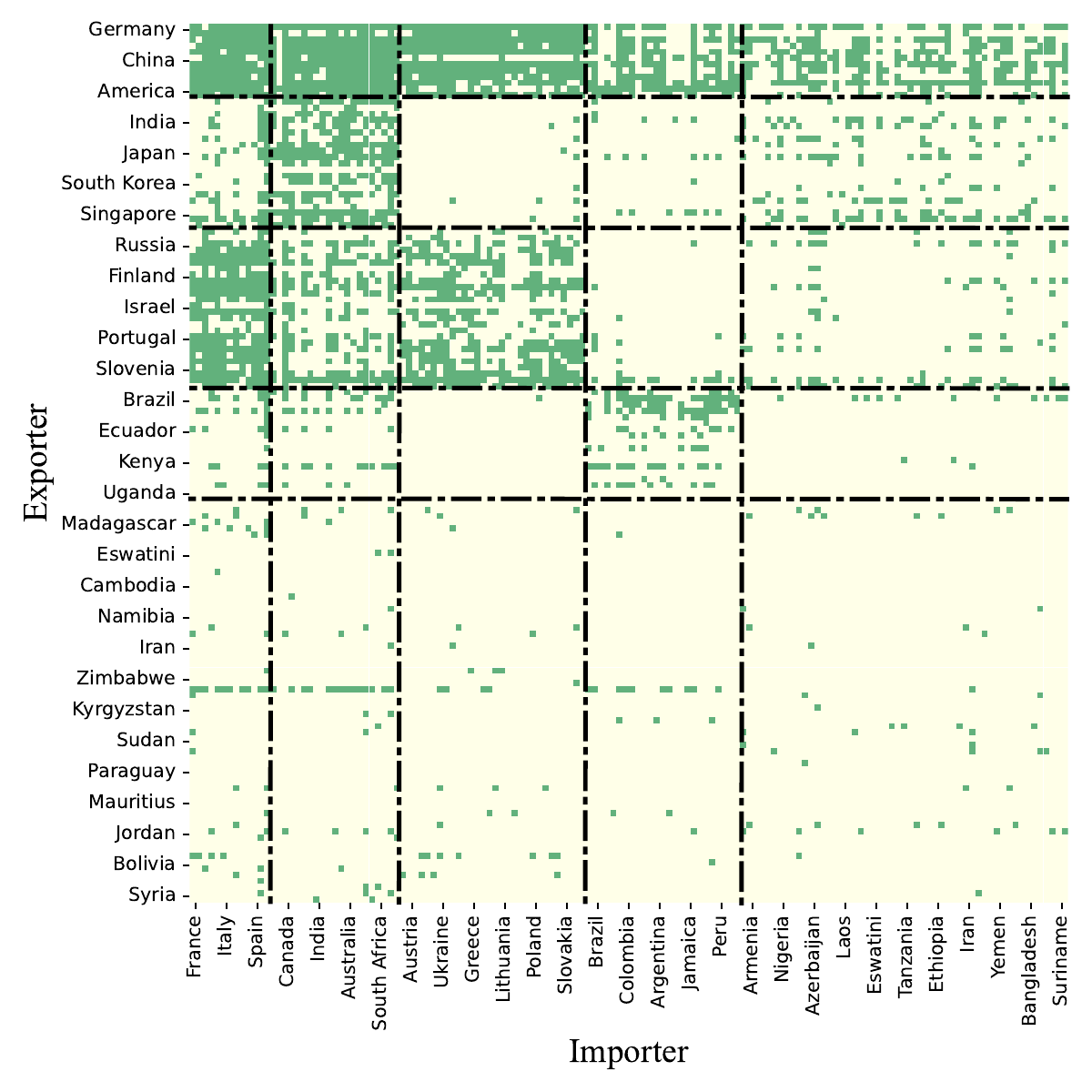}
        \vspace{-0.16\textwidth}
        \caption{Chocolate products nes}
    \end{subfigure}
    \hspace{0.05\textwidth}
    \begin{subfigure}[b]{0.37\textwidth}
        \includegraphics[width=\textwidth]{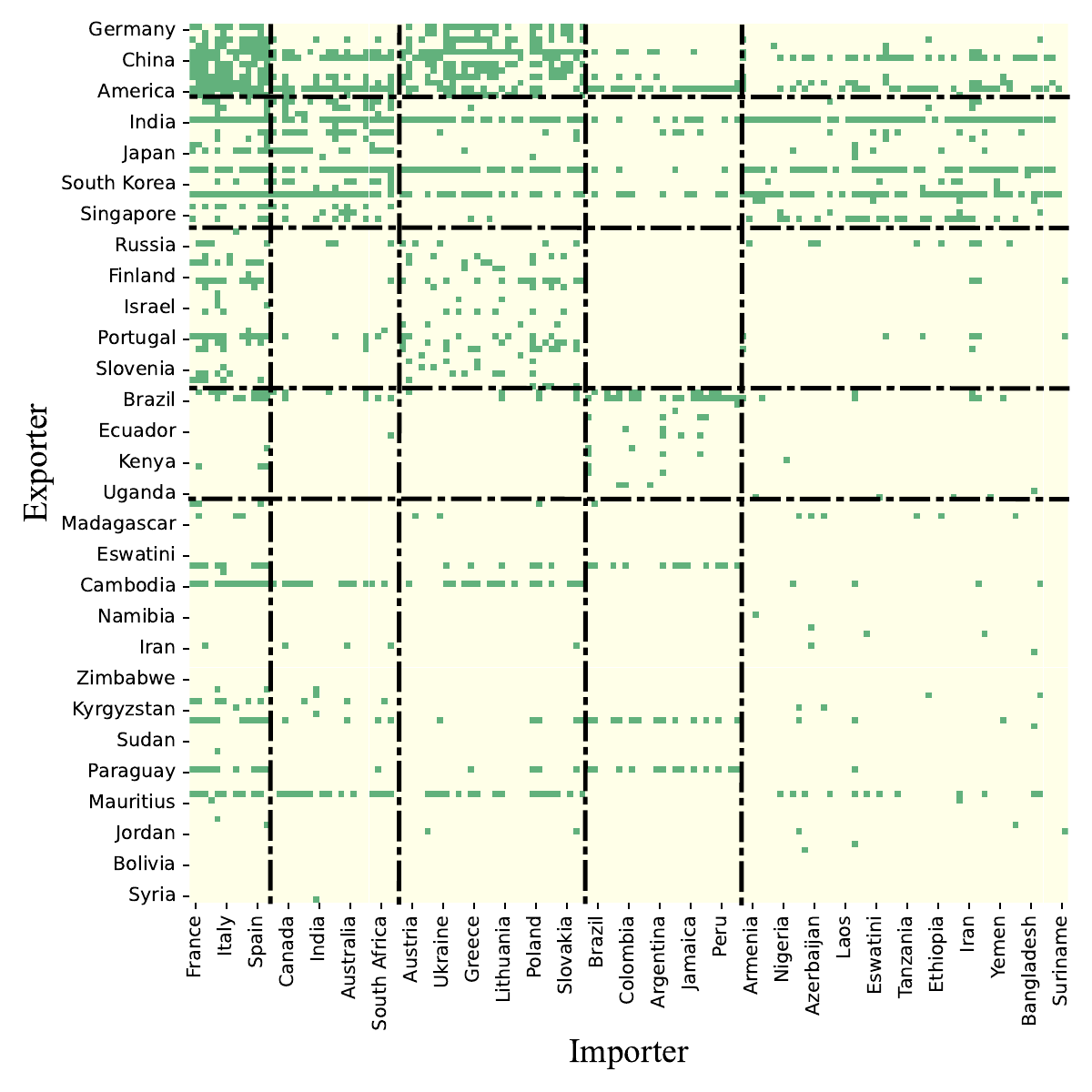}
        \vspace{-0.16\textwidth}
        \caption{Rice, paddy}
    \end{subfigure}
    
    \caption{Visualization of the adjacency matrices for six products. The colors, node order, and community partition are identical to Figure \ref{Adj_sum}.}
    \label{Adjs}
\end{figure*}

Visualizations of the aggregated adjacency matrix and several single-layer adjacency matrices are also presented, which can illustrate the advantages of multi-layer networks and further highlight the difference between sending and receiving patterns. Figure \ref{Adj_sum} illustrates the network of the products listed in Table \ref{items}, aggregated in a summation manner. In Figure \ref{Adjs}, the adjacency matrices for the six products featured in Table \ref{items} are displayed. To highlight the clustering outcomes, the order of the nodes (countries) is adjusted accordingly and persists in each plot. Due to space limitations, only a portion of the countries are labeled on the axes. It is clear that our clustering results have a significant community structure, and the row (export) clusters are quite different from the column (import) clusters, as reflected by the fact that the row and column clusters in the same country have distinctly different compositions and different community sizes. Such asymmetric information cannot be captured by undirected networks. More significantly, a comparison between the individual adjacency matrices in Figure \ref{Adjs} and Figure \ref{Adj_sum} reveals that a single-layer network only encapsulates a fraction of the community structure information, and aggregating the individual trade networks for multiple products provides a more comprehensive picture of trade relations among countries.
This aligns with our original intention of dealing with multi-layer directed networks. 

Finally, we highlight the potential loss of cluster divisions through direct aggregation by applying the \textsf{Sum} method to the WFAT data. We also illustrate that a single-layer network is insufficient for comprehensive cluster information by quantifying layer-wise block probability matrices. See Appendix \ref{AppendixE} for these results. The Python code to run simulations and analysis for this paper is available at \href{https://github.com/WenqingSu/DSoG}{https://github.com/WenqingSu/DSoG}.

\section{Conclusion}\label{conclusion}
In this paper, we studied the problem of detecting co-clusters in multi-layer directed networks. We typically assumed that the multi-layer directed network is generated from the multi-layer ScBM, which allows different row and column clusters via different patterns of sending and receiving edges. The proposed method \textsf{DSoG} is formulated as the spectral co-clustering based on the SoG matrices of the row and column spaces, where a bias correction step is implemented. 
We systematically studied the algebraic properties of the population version of \textsf{DSoG}. In particular, we did not require that the corresponding block probability matrix be of full rank. We also studied the misclassification error rates of \textsf{DSoG}, which show that under certain condition, multiple layers can bring benefit to the clustering performance. We finally conducted numerical experiments on a number of simulated and real examples to support the theoretical results.

There are many ways to extend the content of this paper. First, we focused on the scenario where the layer-wise community structures are homogeneous. It is of great interest to extend the scenario to associated but inhomogeneous community structures, which has been studied recently for multi-layer undirected networks \citep{han2015consistent, pensky2019spectral, chen2022global}. Second, we analyzed the WFAT dataset with respect to a particular year. However, this dataset contains the annual trading data from 1986 to the present and is continuously updated, which actually constitutes a \emph{higher-order} multi-layer network that contains more adequate information than a single multi-layer network. It is thus of great importance to develop a corresponding toolbox for this higher-order multi-layer network. Third, it is important to theoretically study the estimation of the number of communities \citep{fishkind2013consistent,ma2021determining}.


\appendix

\renewcommand{\thesection}{\Alph{section}}
\renewcommand{\thesubsection}{\Alph{subsection}}



\section*{Appendix} \label{appendix}
Appendix \ref{appendixA} provides the technical theorem and lemma that are needed to prove the misclassification rates. Appendix \ref{appendixB} includes the proof of main theorem and lemmas in the main text. Appendix \ref{AppendixC} contains additional theoretical results. Appendix \ref{AppendixD} contains auxiliary lemmas. Appendix \ref{AppendixE} provides the additional results for simulation and real data analysis.

\subsection{Technical theorem and lemma}\label{appendixA}
\begin{myThm}\label{sparse}
	Let $A_l \in \{0, 1\}^{n\times n}$ be the adjacency matrices generated by a multi-layer ScBM and $X_l = A_l - \bar P_l$ be the asymmetric noise matrices for all $1 \leq l \leq L$, where $\bar P_l = \mathbb{E}(A_l)$. If $L^{1/2}n\rho \geq c_1\log(L+n)$ and $n\rho \leq c_2$ for positive constants $c_1$ and $c_2$, then $E_1$,  the off-diagonal part of $\sum_{l=1}^L X_lX_l^T $, satisfies 
	\begin{equation*}
		\|E_1\|_2 \leq cL^{1/2}n\rho \log^{1/2}(L+n)
	\end{equation*}
	with probability at least $1-O((L+n)^{-1})$ for some constant $c>0$.
\end{myThm}

To handle the complicated dependence in $E_1$ caused by the quadratic form, the key idea in the proof is viewing $E_1$ as a matrix-valued U-statistic with a centered kernel function of order two, indexed by the pairs $(i, j)$, and using the decoupling technique described in Lemma \ref{decoupling}, which reduces problems on dependent variables to problems on related (conditionally) independent variables. This type of proof technique is also used in \cite{LeiJ2022Bias}, dealing with the symmetric case. Specifically, rearrange $E_1$ as
	\begin{equation}\label{E1}
		E_1 = \sum_{l=1}^L \sum_{(i,\;j)\neq(i',\;j')}X_{l,ij}X_{l,i'j'}e_ie_{i'}^TI_{\{j = j'\}}, 
	\end{equation}
	where $e_i$ is the standard basis vector in $\mathbb{R}^n$ and pairs ${(i,j), \;(i',j')} \in \{1,2,\ldots,n\}^2$, which can be viewed as a matrix-valued U-statistic defined on the vectors $X_{1,ij}, \ldots, X_{L,ij}$ indexed by the pairs $(i,j)$. The decoupling technique reduces the problem of bounding $\|E_1\|_2$ to that of bounding $\|\sum_{l=1}^LX_l\widetilde{X}_l^T\|_2$ and $\|\mbox{diag}(\sum_{l=1}^LX_l\widetilde{X}_l^T)\|_2$, where $\widetilde{X}_l$ is an independent copy of $X_l$ for all $1 \leq l \leq L$.

\begin{proof}	
	To use the decoupling argument, define 
	\begin{gather*}
		\widetilde{E}_1 = \sum_{l=1}^L \sum_{(i,\;j)\neq(i',\;j')}X_{l,ij}\widetilde{X}_{l,i'j'}e_ie_{i'}^TI_{\{j = j'\}}, \\ 
		\widetilde{E}_2 = \sum_{l=1}^L \sum_{(i,\;j)}X_{l,ij} \widetilde{X}_{l, ij}e_ie_i^T,
	\end{gather*}
	and 
	\begin{equation*}
		\widetilde{E} = \sum_{l=1}^L X_l\widetilde{X}_l^T = \widetilde{E}_1 + \widetilde{E}_2.
	\end{equation*}
	where $\widetilde{E}_1$ is the zero mean off-diagonal part and $\widetilde{E}_2$ is the diagonal part.
	Note that $\|\widetilde{E}_1\|_2 \leq \|\widetilde{E}\|_2 + \|\widetilde{E}_2\|_2$, we control the spectral norm of $\widetilde{E}$ and $\widetilde{E}_2$ separately.
	
	\noindent \textbf{First step: Controlling $\widetilde{E}$.} Recall that $\widetilde{X}_l = \widetilde{A}_l - \bar P_l$, where $\widetilde{A}_l$ is an independent copy of $A_l$, we reformulate $\widetilde{E}$ as
	\begin{equation*}
		\widetilde{E} = \sum_{l=1}^LX_l\widetilde{X}_l^T = \sum_{l=1}^LX_l\widetilde{A}_l^T - \sum_{l=1}^LX_l\bar P_l^T.
	\end{equation*}
	
	Applying Lemma \ref{N2_lemma} with $(v, b) = (2\rho, 1)$, if $L^{1/2}n\rho^{1/2} \geq c_1\log^{1/2}(L+n)$ for some constant $c_1>0$, we have with probability at least $1-O((L+n)^{-1})$ and universal constant $c>0$,
	\begin{equation}\label{XlPl^T}
		\|\sum_{l=1}^LX_l\bar P_l^T\|_2 \leq cL^{1/2}n^{3/2}\rho^{3/2}\log^{1/2}(n+L).
	\end{equation}  
	Note that the constant $c$ may be different from line to line in this proof.
	
	In order to apply Lemma \ref{N2_lemma} to control $\sum_{l=1}^LX_l\widetilde{A}_l^T$ conditioning on $\widetilde{A}_1, \ldots, \widetilde{A}_L$, we need to upper bound $\max_l\|\widetilde{A}_l^T\|_{2,\infty}$, $\sum_{l=1}^L\|\widetilde{A}_l^T\|_F^2$ and $\|\sum_{l=1}^L\widetilde{A}_l\widetilde{A}_l^T\|_2$, respectively. Note that $\widetilde{A}_l$ is an independent copy of $A_l$, now we bound the corresponding version of $A_l$.
	
	We first consider $\max_l\|A_l^T\|_{2,\infty}$. Note that $\mathbb{E}d_{l,i}^{in} \leq n\rho$. Applying a union bound over $i\in [n]$ and $l\in[L]$, along with the standard Bernstein's inequality
	\begin{equation*}
		\mathbb{P}(d_{l,i}^{in} - \mathbb{E}d_{l,i}^{in} \geq t) \leq \exp\left(-\frac{t^2/2}{n\rho + t}\right)
	\end{equation*}
	for all $t>0$, and using the assumption that $n\rho \leq c_2$, then with probability at least $1-O((L+n)^{-1})$, we have
	\begin{eqnarray}\label{max d_i}
		\max_l\|A_l^T\|_{2,\infty} = \max_{l, i}\sqrt{d_{l, i}^{in}} \leq c\log^{1/2}(L+n).
	\end{eqnarray}
		
	For $\sum_{l=1}^L\|A_l^T\|_F^2$, using Bernstein's inequality for $\sum_{l=1}^Ld_{l,i}^{in}$ and the assumption that $L^{1/2}n\rho \geq c_1\log(L+n)$, we have with probability at least $1-O((L+n)^{-1})$,
	\begin{eqnarray}\label{max d_i sum}
		\sum_{l=1}^L\|A_l^T\|_F^2 \leq n\max_i\sum_{l=1}^Ld_{l,i}^{in} \leq cLn^2\rho.
	\end{eqnarray}
	
	Now we turn to $\|\sum_{l=1}^LA_lA_l^T\|_2$. To begin with, decompose $\sum_{l=1}^LA_lA_l^T$ as $F_1 + F_2$, where $F_1$ is the off-diagonal part, which can be viewed as a matrix-valued U-statistic as in (\ref{E1}), and $F_2$ is the diagonal part. 
	For the off-diagonal part $F_1$, $\widetilde{F}_1$ is defined as the off-diagonal part of $\sum_{l=1}^LA_l\widetilde{A}_l^T$ for the purpose of using the decoupling technique. Using symmetrization and Perron-Frobenius theorem, we have
	\begin{equation*}
		\|\widetilde{F}_1\|_2 =\begin{Vmatrix} 0&\widetilde{F}_1\\ \widetilde{F}_1^T&0 \end{Vmatrix}_2 \leq \max\left\{\|\widetilde{F}_1\|_{1,\infty}, \|\widetilde{F}_1^T\|_{1,\infty}\right\} \leq \max\left\{\left\|\sum_{l=1}^LA_l\widetilde{A}_l^T \right\|_{1,\infty}, \left\|\sum_{l=1}^L\widetilde{A}_lA_l^T \right\|_{1,\infty}\right\}.
	\end{equation*} 
	By simple calculations, the $l_1$ norm of the $i$th row of $\sum_{l=1}^LA_l\widetilde{A}_l^T$ is $\sum_{l=1}^L\sum_{j=1}^nA_{l,ij}\widetilde{d}_{l,j}^{in}$. 
Combining (\ref{max d_i}) and (\ref{max d_i sum}) with Bernstein's inequality, we have
	\begin{eqnarray*}
		& &\mathbb{P}\left(\left.\sum_{l=1}^L\sum_{j=1}^nA_{l,ij}\widetilde{d}_{l,j}^{in} - \sum_{l=1}^L\sum_{j=1}^n\mathbb{E}(A_{l,ij})\widetilde{d}_{l,j}^{in} \geq t \right|\widetilde{A}_1,\ldots, \widetilde{A}_L\right)\\
		&\leq & \exp\left(-\frac{t^2/2}{cLn^2\rho^2\log(L+n) + c\log(L+n)t} \right)
	\end{eqnarray*}
	for all $t>0$. By applying the total probability theorem, the union bound, and  the assumption that $L^{1/2}n\rho \geq c_1\log(L+n)$, we have with high probability, $\max_i\sum_{l=1}^L\sum_{j=1}^nA_{l,ij}\widetilde{d}_{l,j}^{in} \leq cL^{1/2}n\rho\log(L+n) \leq cLn^2\rho^2$. A similar argument is conducted for $\sum_{l=1}^L\widetilde{A}_lA_l^T $. By Lemma \ref{decoupling} we have 
	\begin{equation}\label{F1}
		\|F_1\|_2 \leq cLn^2\rho^2
	\end{equation}
	with high probability. For the diagonal part $F_2$, applying (\ref{max d_i sum}), we have 
	\begin{equation}\label{F2}
		\|F_2\|_2 = \max_i\sum_{l=1}^Ld_{l,i}^{out} \leq cLn\rho
	\end{equation}
	 with high probability. 
  Combining (\ref{F1}), (\ref{F2}) with the assumption that $n\rho \leq c_2$, we have with probability at least $1-O((L+n)^{-1})$,
	\begin{equation*}
		\left\|\sum_{l=1}^LA_lA_l^T\right\|_2 \leq cLn\rho.
	\end{equation*}
	So combining the bounds for $\max_l\|\widetilde{A}_l^T\|_{2,\infty}$, $\sum_{l=1}^L\|\widetilde{A}_l^T\|_F^2$ and $\|\sum_{l=1}^L\widetilde{A}_l\widetilde{A}_l^T\|_2$, and applying Lemma \ref{N2_lemma}, we have with probability at least $1-O((L+n)^{-1})$,
	\begin{equation*}
		\left\|\sum_{l=1}^LX_l\widetilde{A}_l^T\right\|_2 \leq cL^{1/2}n\rho\log^{1/2}(L+n).
	\end{equation*}
	
	Combining this with (\ref{XlPl^T}), we have with probability at least $1-O((L+n)^{-1})$, 
	\begin{equation*}
		\|\widetilde{E}\|_2 \leq cL^{1/2}n\rho\log^{1/2}(L+n).
	\end{equation*}
	
	\noindent \textbf{Second step: Controlling $\widetilde{E}_2$.} Recall that $\widetilde{E}_2$ is a diagonal matrix whose $i$th diagonal element is $\sum_{l=1}^L\sum_{j=1}^nX_{l,ij}\widetilde{X}_{l,ij}$. Using standard Bernstein's inequality and union bound, we have with  probability at least $1-O((L+n)^{-1})$,
	\begin{equation*}
		\|\widetilde{E}_2\|_2 \leq cL^{1/2}n\rho\log^{1/2}(L+n).
	\end{equation*}
	
	The claim follows by combining the bounds for $\|\widetilde{E}\|_2$ and $\|\widetilde{E}_2\|_2$ together with the decoupling inequality in Lemma \ref{decoupling}. 
\end{proof}

\begin{myLe}\label{kth eigenvalue}
	Let $H$ be an $n\times s$ matrix with full column rank and the $s$th eigenvalue of $HH^T$ is at least $h$ for some constant $h>0$. Let $G$ be an $s\times s$ symmetric matrix with ${\rm rank}(G) = j \leq s$ and the $j$th eigenvalue of $G$ is at least $g$ for some constant $g>0$. Then $\lambda_j(HGH^T) \geq hg$.
\end{myLe}
\begin{proof}
	Note that $\mbox{rank}(H) = s$, we have $\mbox{dim}\{x \in \mathbb R^n~|~x\perp \mbox{null}(H^T)\} = s$, where $\mbox{null}(H^T)$ is the nullspace of $H^T$ and is given by  $\mbox{null}(H^T):=\{\theta \in \mathbb R^n~|~ H^T\theta=0\}$. Combining this with $\mbox{rank}(G) = j$, we have $\mbox{dim}\{x\in \mathbb R^n~|~x\perp \mbox{null}(H^T),\;H^Tx\perp\mbox{null}(G) \} = s$. Thus 
	\begin{equation*}
		\min_{x\in \mathbb R^n, \; x\perp \mbox{null}(H^T),\;H^Tx\perp \mbox{null}(G)} \frac{x^THGH^Tx}{x^Tx} \geq g\cdot \min_{x\in \mathbb R^n, \;x\perp \mbox{null}(H^T)}\frac{x^THH^Tx}{x^Tx} \geq hg.
	\end{equation*}
	Applying the Courant-Fischer minimax theorem (Theorem 8.1.2 of \cite{golub2013matrix}), intersecting on this event, we have
	\begin{equation*}
		\lambda_j(HGH^T)=\max_{\mbox{dim}(\mathscr{H})=j} \min_{0 \neq x \in \mathscr{H}} \frac{x^THGH^Tx}{x^Tx} \geq hg,
	\end{equation*}
	where $\mathscr{H}$ is the subspace of $\mathbb R^n$. The proof is completed.
\end{proof}

\subsection{Main proofs}\label{appendixB}

\subsubsection*{Proof of Theorem \ref{row_thm}}
We decompose the matrix $S^R$ into the sum of a signal term and noise terms. 

Recall that $\mathcal P_l = \rho YB_lZ^T$ and $\bar P_l = \mathcal P_l - \mbox{diag}(\mathcal P_l)$.
Define 
\begin{equation}\label{N4}
	N_4 = \mbox{diag}(\sum_{l=1}^LX_lX_l^T) - \sum_{l=1}^LD_l^{out},
\end{equation}
and recall the definition of $S^R$ in (\ref{S^R}) and the decomposition (\ref{sumAA^T}), we have 
\begin{equation*}
	S^R = \sum_{l=1}^L\mathcal P_l\mathcal P_l^T + N_1 + N_2 + N_3 + N_4.
\end{equation*}

We first control the signal term. Recall that $\Delta_z = \mbox{diag}(\sqrt{n_1^z}, \ldots, \sqrt{n_{K_z}^z})$ is a $K_z \times K_z$ diagonal matrix, then $Z = \widetilde{Z}\Delta_z$ with $\widetilde{Z}$ being a column orthogonal matrix. By the balanced community sizes assumption and the number of column clusters $K_z$ is fixed, the minimum eigenvalue of $\Delta_z$ is lower bounded by $c_0n^{1/2}$ for some constant $c_0>1$. Then we have
\begin{equation}\label{C7_1}
	\sum_{l=1}^L\mathcal P_l\mathcal P_l^T = \rho^2 \sum_{l=1}^L Y B_l \Delta_z \widetilde{Z}^T \widetilde{Z} \Delta_z B_l^T Y^T  \succeq  c_0n\rho^2 Y \left[\sum_{l=1}^LB_lB_l^T \right] Y^T, 
\end{equation}
where $\succeq$ denotes the Loewner partial order, in particular, let $A$ and $B$ be two Hermitian matrices of order $p$, we say that $A\succeq B$ if $A - B$ is positive semi-definite. Note that $Y$ is of full column rank and the $K_y$th (and smallest) non-zero eigenvalue of $YY^T$ is lower bounded by $c_0n$ for some constant $c_0$, where we used the balanced community sizes assumption and the number of row clusters $K_y$ is fixed. Using Lemma \ref{kth eigenvalue} and Assumption \ref{asmp2}, we can lower bound the $K$th eigenvalue of $\sum_{l=1}^L\mathcal P_l\mathcal P_l^T$ to be
\begin{equation}\label{C7_2}
	\lambda_K(\sum_{l=1}^L\mathcal P_l\mathcal P_l^T) \geq cLn^2\rho^2
\end{equation}
for some constant $c>0$. Note that we use $c$, $c_0$, $c_1$ and $c_2$ to represent the generic constants and they may be different from line to line.

Then we bound each noise term respectively. The first noise term $N_1$ is non-random and satisfies
\begin{equation*}
	\|N_1\|_2 \leq n\|N_1\|_{\max} \leq Ln\rho^2.
\end{equation*}

For $N_2$, recalling that $X_l = A_l - \bar P_l$, where $X_{l,ij}(1\leq i,j \leq n)$ is generated independently from centered Bernoulli, and is therefore $(2\rho, 1)$-Bernstein. Using Lemma \ref{N2_lemma} and the fact that $\|\bar P_l^T \bar P_l\|_2 \leq n^2\rho^2$, $\|\bar P_l\|_2\leq n^2\rho^2$ and $\|\bar P_l\|_{2, \infty}\leq n^{1/2}\rho$, if $L^{1/2}n\rho^{1/2} \geq c_1\log(L+n)$ for some constant $c_1>0$, we have with probability at least $1-O((L+n)^{-1})$ and universal constant $c>0$,
\begin{equation*}
	\|N_2\|_2 \leq cL^{1/2}n^{3/2}\rho^{3/2}\log^{1/2}(n+L).
\end{equation*}  

For the noise term $N_3$, applying Theorem \ref{sparse}, we have 
\begin{equation*}
	\|N_3\|_2 \leq cL^{1/2}n\rho \log^{1/2}(L+n)
\end{equation*}
with probability at least $1-O((L+n)^{-1})$.

We next control $N_4$. The construction in (\ref{S2ii}) implies that
\begin{equation*}
	\|N_4\|_2 \leq Ln\rho^2.
\end{equation*}

Let $U$ and $\widehat{U}$ be the $n\times K$ matrices consisting of the leading eigenvectors of $\sum_{l=1}^L\mathcal P_l\mathcal P_l^T$ and $S^R$, respectively. We are now ready to bound the derivation of $\widehat{U}$ from $U$. 	
Combining the lower bound of signal term and upper bound of all noise terms, we have with high probability,
\begin{eqnarray*}
	\frac{\|N_1 + N_2 + N_3 + N_4\|_2}{\lambda_K (\sum_{l=1}^L\mathcal P_l\mathcal P_l^T)} &\leq& c \ \frac{Ln\rho^2 + L^{1/2}n\rho\log^{1/2}(L+n)}{Ln^2\rho^2} \\
	& \leq & \frac{c}{n} + \frac{c\log^{1/2}(L+n)}{L^{1/2}n\rho},
\end{eqnarray*}
where the first inequality arises from the merging of the $N_2$ and $N_3$ terms when $n\rho \leq c_2$ for some constant $c_2>0$. By Proposition 2.2 of \cite{vu2013minimax} and Davis-Kahan sin$\Theta$ theorem (Theorem VII.3.1 of \cite{Bhatia1997matrix}), there exists a $K \times K$ orthogonal matrix $O$ such that
\begin{eqnarray*}
	\|\widehat{U} - UO\|_F & \leq & \sqrt{K}\|\widehat{U} - UO\|_2 \\[6pt]
	& \leq & \frac{\sqrt{K}\|N_1 + N_2 + N_3 + N_4\|_2}{\lambda_K (\sum_{l=1}^L\mathcal P_l\mathcal P_l^T) - \|N_1 + N_2 + N_3 + N_4\|_2} \\[6pt]
	& \lesssim & \frac{1}{n}+\frac{\log^{1/2}(L+n)}{L^{1/2}n\rho},
\end{eqnarray*}
where $a_n \lesssim b_n$ means that there exists some positive constant $c$ such that $a_n \leq cb_n$ for all $n$. 
Finally, by using Lemma \ref{kmeans}, we are able to obtain the desired bound for the misclassification rate. $\hfill\qedsymbol$

\subsubsection*{Proof of Lemma \ref{row interpretation}}
Define $\Delta_y = \mbox{diag}(\sqrt{n_1^y}, \ldots, \sqrt{n_{K_y}^y})$ as a $K_y \times K_y$ diagonal matrix with $k$th diagonal entry being the $l_2$ norm of the $k$th column of $Y$. Then $Y\Delta_y^{-1}$ is a column orthogonal matrix. Similarly define $\Delta_z = {\rm diag}(\sqrt{n_1^z}, \ldots, \sqrt{n_{K_z}^z})$. Write $\mathcal P^R$ as
\begin{eqnarray*}
	\mathcal P^R & = & \rho^2Y \sum_{l=1}^LB_lZ^TZB_l^TY^T \\
	&=& \rho^2Y\Delta_y^{-1}\Delta_y\sum_{l=1}^LB_l\Delta_z^2B_l^T\Delta_y\Delta_y^{-1}Y^T = \rho^2Y\Delta_y^{-1}Q^RD^R{Q^R}^T\Delta_y^{-1}Y^T,
\end{eqnarray*}
where we used the eigendecomposition of $\Delta_y \sum_{l=1}^L B_l\Delta_z^2B_l^T \Delta_y$ is $Q^RD^R{Q^R}^T$. Here $Q^R$ is a $K_y\times K$ column orthogonal matrix and $D^R$ is a $K\times K$ diagonal matrix, which is due to the fact that $Z^TZ$ is a positive definite diagonal matrix, and hence the rank of $\sum_{l=1}^LB_lB_l^T$ is equal to the rank of $\Delta_y \sum_{l=1}^L B_l\Delta_z^2B_l^T \Delta_y$. 

 It is easy to see that $Y\Delta_y^{-1}Q^R$ has orthogonal columns, so we have
\begin{equation}\label{U=YlambdaQ^R}
	U = Y\Delta_y^{-1}Q^R.
\end{equation}

When $\sum_{l=1}^LB_lB_l^T$ is of full rank, that is, $K = K_y$, $\Delta_y^{-1}Q^R$ is invertible, thus $Y_{i\ast} = Y_{j\ast}$ if and only if $U_{i\ast} = U_{j\ast}$. The first claim follows by the fact that the rows of $\Delta_y^{-1}Q^R$ are perpendicular to each other and the $k$th row has length $\sqrt{n_k^y}$.

When $\sum_{l=1}^LB_lB_l^T$ is rank-deficient, that is, $K < K_y$, $Y_{i\ast} = Y_{j\ast}$ can also imply $U_{i\ast} = U_{j\ast}$ by (\ref{U=YlambdaQ^R}). On the other hand, by (\ref{U=YlambdaQ^R}), $\|U_{i\ast} - U_{j\ast}\|_2 := \|\frac{Q^R_{g_i^y\ast}}{\sqrt{n_{g_i^y}}} - \frac{Q^R_{g_j^y\ast}}{\sqrt{n_{g_j^y}}}\|_2$. The second claim follows by the rows of $\Delta_y^{-1}Q^R$ are mutually distinct with their minimum Euclidean distance being larger than a deterministic sequence $\{\zeta_n\}_{n\geq 1}$. $\hfill\qedsymbol$

\subsubsection*{Proof of Lemma \ref{row separable condition}}
By the balanced community sizes assumption and both $K_y$ and $K_z$ are fixed, we have	
 \begin{equation}\label{D upper}
 	\Delta_y\sum_{l=1}^LB_l\Delta_z^2B_l^T\Delta_y \preceq \frac{c_0}{K_z}n\Delta_y\sum_{l=1}^LB_lB_l^T\Delta_y \preceq c_0K_yLn\Delta_y^2 \preceq c_0^2Ln^2I_{K_y},
 \end{equation} 
 where we used $\Delta_z^2 \preceq \frac{c_0n}{K_z}I_{K_z}$, $\Delta_y^2 \preceq \frac{c_0n}{K_y}I_{K_y}$ and $\|\sum_{l=1}^L B_lB_l^T\|_2 \leq LK_yK_z$. Recall that the eigendecomposition of $\Delta_y \sum_{l=1}^L B_l\Delta_z^2B_l^T \Delta_y$ is $Q^RD^R{Q^R}^T$, (\ref{D upper}) implies that $D^R_{kk}$ is upper bounded by $c_0^2Ln^2$ for any $1\leq k \leq K$. 
  
 Define $\mathbb{B}^R := \sum_{l=1}^L B_l\Delta_z^2B_l^T$, by simple calculations, we have 
 \begin{equation*}
 	\mathbb{B}^R_{g_i^yg_j^y} = \sum_{k_z = 1}^{K_z}\sum_{l=1}^Ln_{k_z}^zB_{l,g_i^yk_z}B_{l,g_j^yk_z}
 \end{equation*}
 for all $1 \leq i, j \leq n$. Under the Assumption \ref{asmp1}, it is easy to see that	 
 \begin{equation*}
 	\mathbb{B}^R_{g_i^yg_j^y} \geq \frac{n}{c_0K_z}\left[\sum_{l=1}^LB_lB_l^T \right]_{g_i^yg_j^y}
 \end{equation*}
 and 
 \begin{equation*}
 	\mathbb{B}^R_{g_i^yg_j^y} \leq \frac{c_0n}{K_z}\left[\sum_{l=1}^LB_lB_l^T \right]_{g_i^yg_j^y}.
 \end{equation*}
 Let $\mu^r_n := c_0^2Ln^2$, by the decomposition (\ref{U=YlambdaQ^R}), we have
 \begin{eqnarray*}
 	\mu^r_n\|U_{i\ast} - U_{j\ast}\|_2^2 & = & \sum_{k=1}^K \mu^r_n (\frac{Q^R_{g_i^y k}}{\sqrt{n_{g_i^y}}}  - \frac{Q^R_{g_j^y k}}{\sqrt{n_{g_j^y}}})^2 \\ 
 	&\geq & \sum_{k=1}^K D^R_{kk} (\frac{Q^R_{g_i^y k}}{\sqrt{n_{g_i^y}}}  - \frac{Q^R_{g_j^y k}}{\sqrt{n_{g_j^y}}})^2 \\
 	& = & \sum_{k=1}^K D^R_{kk} (\frac{Q^R_{g_i^y k}}{\sqrt{n_{g_i^y}}})^2 + \sum_{k=1}^K D^R_{kk} (\frac{Q^R_{g_j^y k}}{\sqrt{n_{g_j^y}}})^2 - 2\sum_{k=1}^K D^R_{kk}\frac{Q^R_{g_i^y k}Q^R_{g_j^y k}}{\sqrt{n_{g_i^y}n_{g_j^y}}} \\[6pt]
 	& = & \mathbb{B}^R_{g_i^y g_i^y} + \mathbb{B}^R_{g_j^y g_j^y} - 2\mathbb{B}^R_{g_i^y g_j^y}\\
 	& \geq & \frac{n}{c_0K_z} \left[\sum_{l=1}^LB_lB_l^T \right]_{g_i^yg_i^y} + \frac{n}{c_0K_z} \left[\sum_{l=1}^LB_lB_l^T \right]_{g_j^yg_j^y} - 2\frac{c_0n}{K_z} \left[\sum_{l=1}^LB_lB_l^T \right]_{g_i^yg_j^y}\\
 	& \geq & c_0^2Ln^2\zeta_n^2
 \end{eqnarray*}
 for any $Y_{i\ast} \neq Y_{j\ast}$. $\hfill\qedsymbol$

\subsubsection*{Proof of Lemma \ref{column interpretation}}
By the fact that the rank of $\sum_{l=1}^LB_l^TB_l$ is equal to the rank of $\Delta_z \sum_{l=1}^L B_l^T\Delta_y^2B_l \Delta_z$, we have $Q^C$ is a $K_z\times K $ column orthogonal matrix. 	It is easy to see that $Z\Delta_z^{-1}Q^C$ has orthogonal columns, so we have
\begin{equation*}\label{U=YlambdaQ^C}
	V = Z\Delta_z^{-1}Q^C.
\end{equation*}
The rest proof is similar to that of Lemma \ref{row interpretation}, we omit it here. $\hfill\qedsymbol$

\subsubsection*{Proof of Lemma \ref{column separable condition}}
The proof of Lemma \ref{column separable condition}
follows the same strategy as that of Lemma \ref{row separable condition}. Here we only describe the difference.

Define $\mu^c_n := c_0^2Ln^2$ and $\mathbb{B}^C := \sum_{l=1}^L B_l^T\Delta_y^2B_l$, by simple calculations, we have 
 \begin{equation*}
 	\mathbb{B}^C_{g_i^zg_j^z} = \sum_{k_y=1}^{K_y}\sum_{l=1}^Ln_{k_y}^zB_{l,k_yg_i^z}B_{l,k_yg_j^z}
 \end{equation*}
 for all $1 \leq i, j \leq n$. Combining this with Assumption \ref{asmp1} and $V = Z\Delta_z^{-1}Q^C$, we have
 \begin{eqnarray*}
 	\mu^c_n\|V_{i\ast} - V_{j\ast}\|_2^2 
 	& \geq & \sum_{k=1}^{K^\prime} D^C_{kk} (\frac{Q^C_{g_i^z k}}{\sqrt{n_{g_i^z}}}  - \frac{Q^C_{g_j^z k}}{\sqrt{n_{g_j^z}}})^2 \\[6pt]
 	& = & \mathbb{B}^C_{g_i^z g_i^z} + \mathbb{B}^C_{g_j^z g_j^z} - 2\mathbb{B}^C_{g_i^z g_j^z}\\[6pt]
 	& \geq & \frac{n}{c_0K_y} \left[\sum_{l=1}^LB_l^TB_l \right]_{g_i^zg_i^z} +  \frac{n}{c_0K_y} \left[\sum_{l=1}^LB_l^TB_l \right]_{g_j^zg_j^z} - 2 \frac{c_0n}{K_y}\left[\sum_{l=1}^LB_l^TB_l \right]_{g_i^zg_j^z}\\
 	& \geq & c_0^2Ln^2\xi_n^2
 \end{eqnarray*}
 for any $Z_{i\ast} \neq Z_{j\ast}$.  $\hfill\qedsymbol$

\subsection{Additional theoretical results}\label{AppendixC}
Proposition \ref{dense_thm} provides the misclassification rate of \textsf{DSoG} under a dense regime with $n\rho \geq c_1\log(L+n)$ for some constant $c_1>0$. Proposition \ref{sog_thm} provides the misclassification rate of \textsf{SoG} algorithm, namely, the \textsf{DSoG} without the bias-adjustment step.

\begin{myProp}\label{dense_thm}
    Suppose that Assumptions {\rm \ref{asmp1}} and {\rm \ref{asmp2}}, and {\rm (\ref{row separable})} hold. If   $n\rho \geq c_1\log(L+n)$ for some constant $c_1>0$, then the output $\widehat{Y}$ of Algorithm {\rm \ref{alg1}} satisfies  
	\begin{equation*}
		\mathcal{L}(Y,\widehat{Y}) \leq \frac{c}{n{\zeta_n}^2}\left(\frac{1}{n^2} + \frac{\log(L+n)}{Ln\rho}\right)
	\end{equation*}
with probability at least $1-O((L+n)^{-1})$ for some constant $c>0$.
\end{myProp}
The proof of this bound is similar to that of Theorem {\rm \ref{row_thm}}, except that Theorem {\rm \ref{sparse}} used to control the noise term $N_3$ is not applicable (alternatively, Theorem 4 of \cite{LeiJ2022Bias} can be used), and the upper bound of the noise term $N_2$  becomes the dominant term.

\begin{myProp}\label{sog_thm}
    Under the same conditions as in Theorem {\rm \ref{row_thm}}, let $\widehat{Y}$ denote the estimated membership matrices with respect to the row clusters of the \textsf{SoG} method, we have 
    \begin{equation*}
		\mathcal{L}(Y,\widehat{Y}) \leq \frac{c}{n{\zeta_n}^2}\left(\frac{1}{n^2\rho^2} + \frac{\log(L+n)}{Ln^2\rho^2}\right)
	\end{equation*}
	with probability at least $1-O((L+n)^{-1})$ for some constant $c>0$.
\end{myProp}
The proof of this bound is similar to that of Theorem {\rm \ref{row_thm}}, except that the upper bound of the noise term $N_4$ in \eqref{N4} now satisfies $\|N_4\|_2 \leq c_1Ln\rho$ with high probability for some positive constant $c_1$. The primary improvement of \textsf{DSoG} over \textsf{SoG} is observed in the first term of the upper bound. As a result, when the network is sparse, namely, $\rho$ is small, the bias-adjustment step resulted in a significant reduction in the upper bound of the misclassification rate.

\subsection{Auxiliary lemmas}\label{AppendixD}
Given a random variable $X$, we say that Bernstein’s condition with parameters $v$ and $b$ holds if $\mathbb{E}[|X|^k] \leq \frac{v}{2}k!b^{k-2}$ for all integers $k\geq 2$. It is also said that $X$ is $(v, b)$-Bernstein.

\begin{myLe}[Theorem 3 in \cite{LeiJ2022Bias}]\label{N2_lemma}
	For $1 \leq l \leq L$, let $W_l$ be a sequence of independent $n\times n$ matrices with zero mean independent entries, and $H_l$ be any sequence of $n\times n$ non-random matrices. If for all $1 \leq l \leq L$ and $1 \leq i, j \leq n$, each entry $W_{l,ij}$ is $(v, b)$-Bernstein, then for all $t > 0$,
	\begin{equation*}
		\mathbb{P}\left( \| \sum_{l=1}^L W_lH_l\|_2 \geq t\right) \leq 4n\exp\left(-\frac{t^2/2}{v\max(n\|\sum_{l=1}^L H_l^T H_l\|_2,\;\sum_{l=1}^L\|H_l\|_F^2) + b\max_l\|H_l\|_{2, \infty}t} \right).
	\end{equation*}
\end{myLe}

\begin{myLe}[Lemma 5.3 in \cite{lei2015consistency}]\label{kmeans}
	Let $U$ be an $n\times d$ matrix with $K$ distinct rows with minimum pairwise Euclidean norm separation $\gamma>0$. Let $\widehat{U}$ be another $n\times d$ matrix and $(\widehat{\Theta}, \widehat{X})$ be an solution to $k$-means problem with input $\widehat{U}$, then the number of errors in $\widehat{\Theta}$ as an estimate of the row clusters of $U$ is no larger than $c\|\widehat{U} - U\|_F^2\gamma^{-2}$ for some constant $c>0$.
\end{myLe}

\begin{myLe}[Theorem 1 in \cite{de1995decoupling}]\label{decoupling}
	Let $\{X_i\}_{i=1}^n$ be a sequence of independent random variables in a measurable space $(\mathcal{S}, \mathscr{S})$, and let $\{X_i^{(j)}\}$, $j=1,\ldots,k$ be $k$ independent copies of $\{X_i\}$. Let $f_{i_1i_2\ldots i_k}$ be families of functions of $k$ variables taking $(\mathcal{S}\times\ldots\mathcal{S})$ into a Banach space $(\mathcal{B}, \|\cdot\|)$. Then, for all $n\geq k\geq2,\;t>0$, there exist numerical constant $c_k>0$ depending on $k$ only so that, 
	\begin{eqnarray*}
		& &\mathbb{P}(\|\sum_{1\leq i_1 \neq i_2\neq  \ldots \neq i_k \leq n}f_{i_1i_2\ldots i_k}(X_{i_1}^{(1)}, X_{i_2}^{(1)},\ldots,X_{i_k}^{(1)})\| \geq t)\\
		& \leq & c_k\mathbb{P}(c_k\|\sum_{1\leq i_1 \neq i_2\neq  \ldots \neq i_k \leq n}f_{i_1i_2\ldots i_k}(X_{i_1}^{(1)}, X_{i_2}^{(2)},\ldots,X_{i_k}^{(k)})\| \geq t).
	\end{eqnarray*}
\end{myLe}

\subsection{Additional results for simulation and real data analysis}\label{AppendixE}
This section presents the additional experimental results that are not shown in the main text.
\subsubsection{Additional results for simulations}
\paragraph{\textbf{Sensitivity of tuning parameters.}} In Section \ref{simulation}, the tuning parameters for the proposed method \textsf{DSoG} were all set to their true values. Here, we conduct an experiment to show the sensitivity of the selection of tuning parameters. We consider the following rank-deficient multi-layer ScBM. We fix $L = 50$ and set $B_l = \rho B^{(1)}$ for $l \in \{1,\ldots, L/2\}$, and $B_l = \rho B^{(2)}$ for $l \in \{L/2+1,\ldots, L\}$, with 
	\begin{equation*}
		B^{(1)} = U\begin{bmatrix} 1.2 & 0 & 0 \\ 0 & 0.4 & 0 \\ 0 & 0 & 0 \end{bmatrix}V^T \approx \begin{bmatrix}0.287 & 0.293 & 0.326 \\ 0.293 & 0.930 & 0.246 \\ 0.326 &0.246 & 0.382 \end{bmatrix}
	\end{equation*} 
	and
	\begin{equation*}
		B^{(2)} = U\begin{bmatrix} 1.2 & 0 & 0 \\ 0 & -0.4 & 0 \\ 0 & 0 & 0 \end{bmatrix}V^T \approx \begin{bmatrix}0.113 & 0.507 & 0.074 \\ 0.507 & 0.670 & 0.554 \\ 0.074 & 0.554 & 0.018 \end{bmatrix},
	\end{equation*} 
	where 
	\begin{equation*}
		U = V \approx \begin{bmatrix} 0.408 & 0.467 & -0.784 \\ 0.817 & -0.571 & 0.085 \\ 0.408 & 0.675 & 0.615
		\end{bmatrix}.
	\end{equation*}
We consider $n = 500$ nodes per network across $K_y = 3$ row clusters and $K_z = 3$ column clusters, with row cluster sizes $n_1^y = 150, n_2^y = 150,  n_3^y = 200$ and column cluster sizes $n_1^z = 150,  n_2^z = 200, n_3^z = 150$. In this setup, the true values of the tuning parameters are $K = 2< K_y=3$ and $K' = 2< K_z=3$. In our experiment, we set $\hat{K} = K_y = 3$ and $\hat{K'} = K_z = 3$, which involves incorrect values of the tuning parameters. The misclassification rates of four methods for different values of $\rho$ are shown in Figure \ref{sensitivity}. It can be seen that \textsf{DSoG} continues to outperform other methods, although there is certain performance degradation for all methods.
\begin{figure*}[!htbp]
    \centering
    \begin{subfigure}[b]{0.48\textwidth}
        \includegraphics[width=\textwidth]{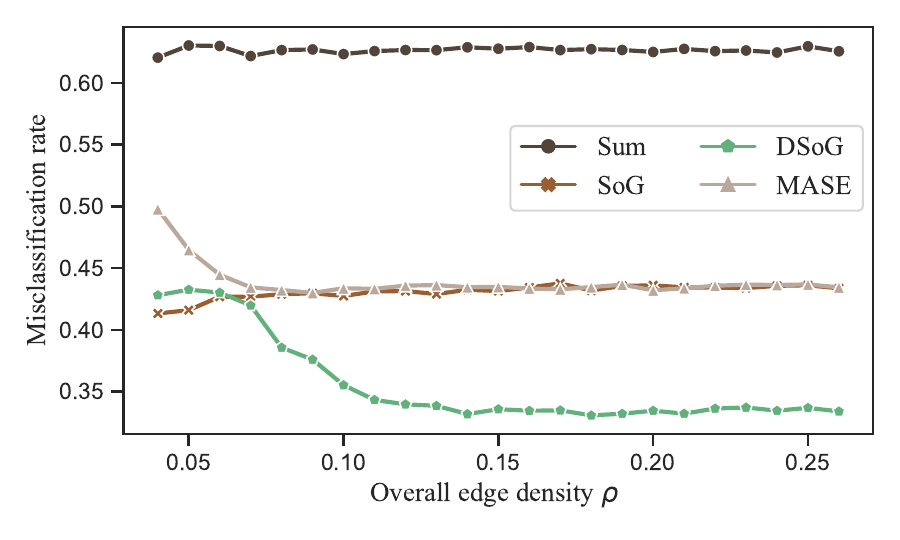}
        \vspace{-0.1\textwidth}
        \caption{row clustering}
    \end{subfigure}\hspace{0.3cm}
    \begin{subfigure}[b]{0.48\textwidth}
        \includegraphics[width=\textwidth]{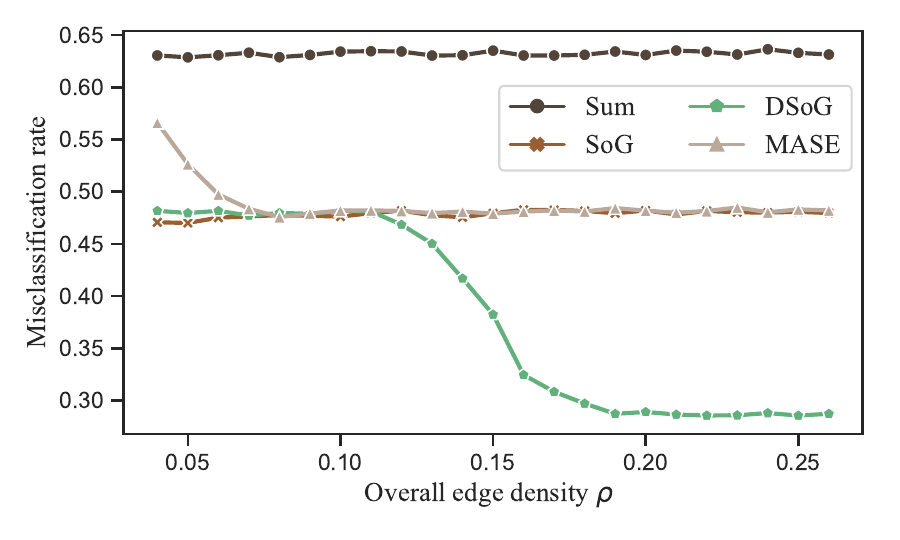}
        \vspace{-0.1\textwidth}
        \caption{column clustering}
    \end{subfigure}\\
    \caption{Misclassification rates of four methods in terms of row and column clustering with incorrect selection of tuning parameters.}
    \label{sensitivity}
\end{figure*}

\paragraph{\textbf{Comparison with likelihood-based methods.}}
In Section \ref{simulation}, the comparison is limited to spectral clustering-based methods. In addition, likelihood-based methods designed for multi-layer networks, as developed in \cite{wang2021fast} and \cite{fu2023profile}, are also capable of handling directed networks. Thus, we compare the performance of likelihood-based methods with our proposed approach in the context of multi-layer ScBMs. Specifically, we consider the model from Experiment 1, with the overall edge density $\rho = 0.1$ and the number of network layers $L = 10$. Both the number of row and column communities are set to 3, with balanced community sizes. The number of nodes $n$ varies from 300 to 1500, and the networks are generated from the multi-layer ScBM. Within this setting, we study two different scenarios. In the first scenario, the row and column communities are different, which is the primary interest in co-clustering. In the second scenario, the row and column communities are identical.

\begin{figure*}[hp]
   \centering
   \begin{subfigure}[b]{0.45\textwidth}
       \includegraphics[width=\textwidth]{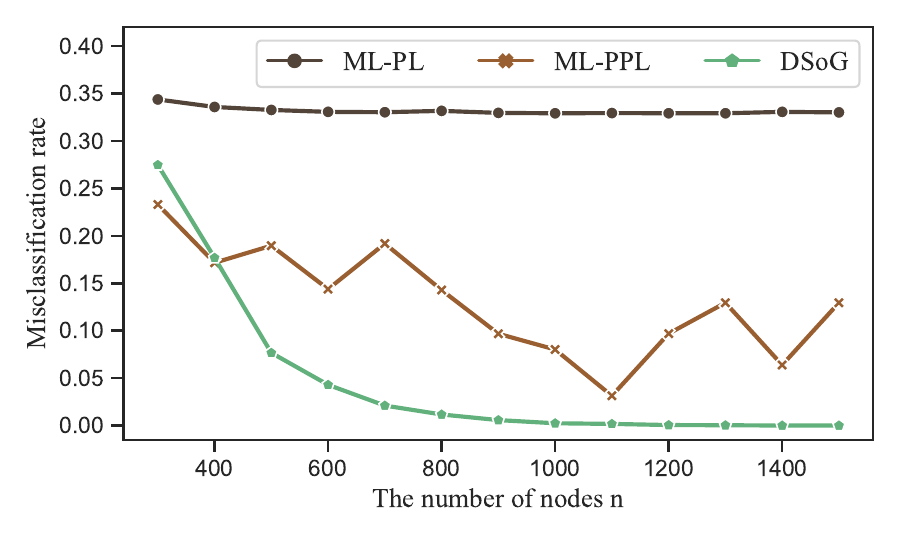}
   \end{subfigure}\hspace{0.5cm}
   \begin{subfigure}[b]{0.45\textwidth}
       \includegraphics[width=\textwidth]{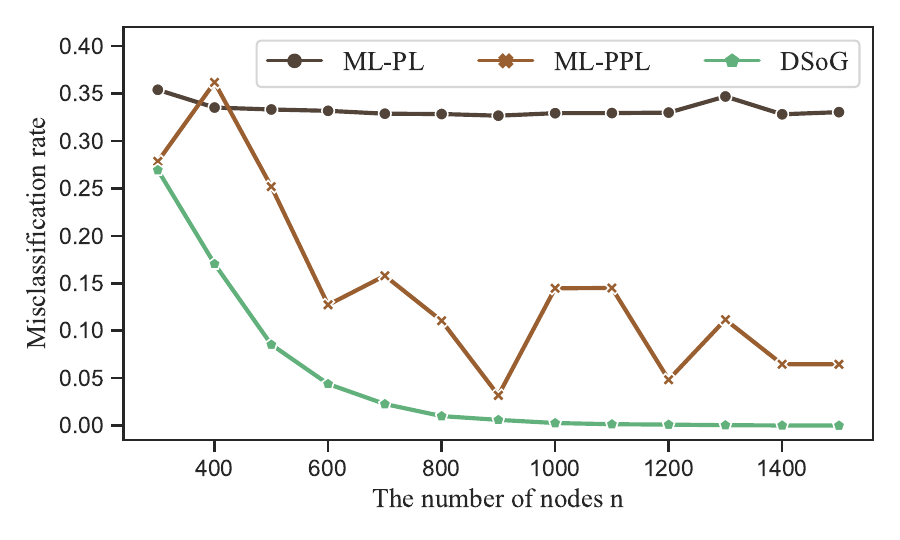}
   \end{subfigure}\\
   \vspace{0.015\textwidth}
   \begin{subfigure}[b]{0.45\textwidth}
       \includegraphics[width=\textwidth]{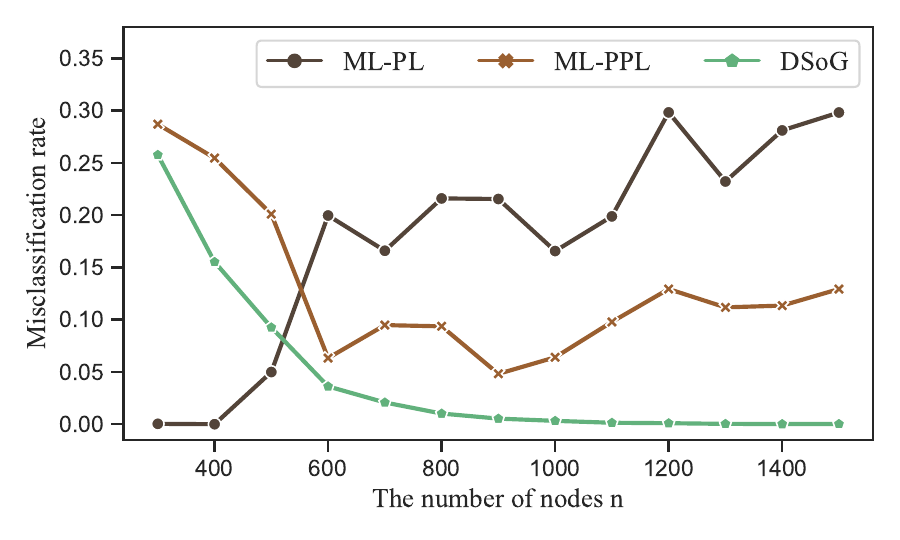}
   \end{subfigure}\hspace{0.6cm}
   \begin{subfigure}[b]{0.45\textwidth}
       \includegraphics[width=\textwidth]{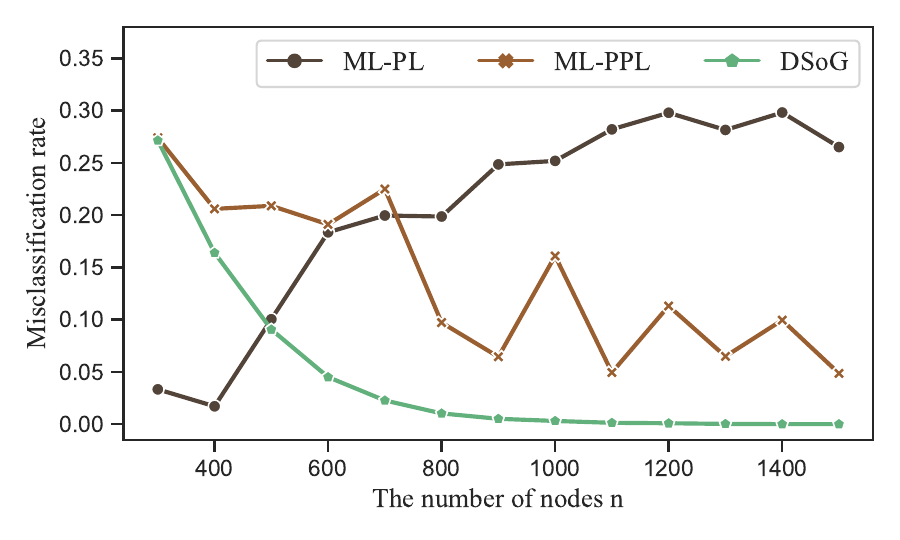}
   \end{subfigure}\\
   \vspace{-0.02\textwidth}
   {\scriptsize (a) \textbf{\textsf{ML-PL} and \textsf{ML-PPL} with \textsf{Sum} Initialization.}} \\
    \vspace{0.025\textwidth}
   \begin{subfigure}[b]{0.45\textwidth}
       \includegraphics[width=\textwidth]{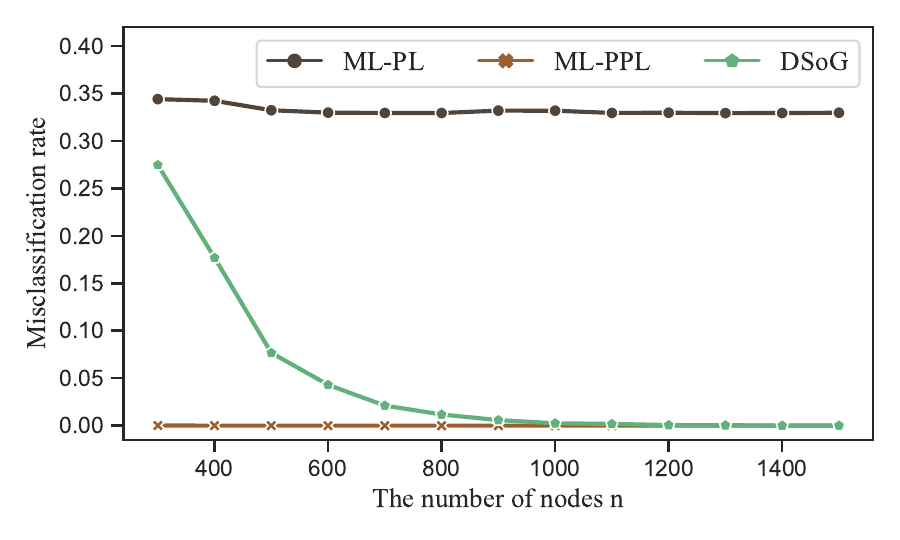}
   \end{subfigure}\hspace{0.6cm}
   \begin{subfigure}[b]{0.45\textwidth}
       \includegraphics[width=\textwidth]{row_dsoginitial.pdf}
   \end{subfigure}\\
   \vspace{0.005\textwidth}
   \begin{subfigure}[b]{0.45\textwidth}
       \includegraphics[width=\textwidth]{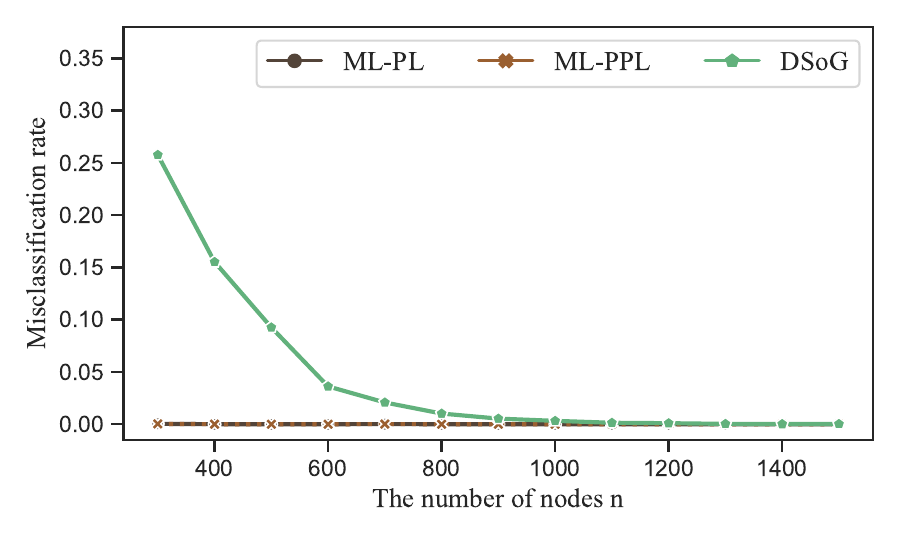}
   \end{subfigure}\hspace{0.6cm}
   \begin{subfigure}[b]{0.45\textwidth}
       \includegraphics[width=\textwidth]{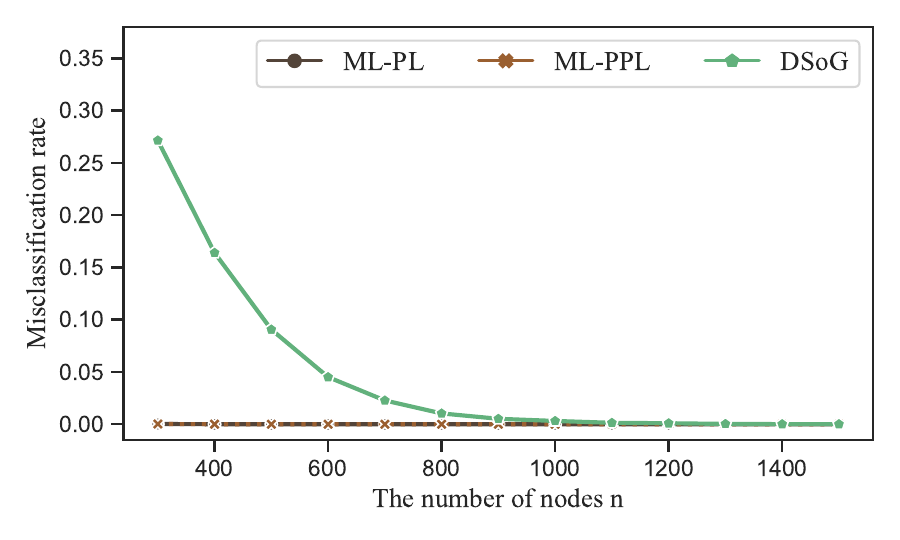}
   \end{subfigure}\\
   {\scriptsize (b) \textbf{\textsf{ML-PL} and \textsf{ML-PPL} with \textsf{DSoG} Initialization.}} \\
	\vspace{-0.02\textwidth}
    \caption{Misclassification rates of three methods in terms of row clustering (left panel) and column clustering  (right panel) with varying $n$. The initial values for both \textsf{ML-PL} and \textsf{ML-PPL} are set to the estimates obtained from \textsf{Sum} in (a) and from \textsf{DSoG} in (b), respectively. In both (a) and (b), the row and column membership matrices of the underlying multi-layer ScBM are different (top row) and identical (bottom row).}
    \label{likelihood based dsog}
\end{figure*}

Since likelihood-based methods are sensitive to initialization, we conduct simulations using two kinds of initial values for these likelihood-based methods. One is the spectral-based method \textsf{Sum} which is suggested in \cite{wang2021fast} and \cite{fu2023profile}, and the other is our method \textsf{DSoG}. The averaged misclassification rates over 50 replications are presented in Figure \ref{likelihood based dsog}, where \textsf{ML-PL} denotes the pseudo-likelihood-based algorithm in \cite{wang2021fast}, and \textsf{ML-PPL} denotes the profile-pseudo likelihood-based method in \cite{fu2023profile}. It is observed that when initializing with \textsf{Sum}, likelihood-based methods are unstable and inferior performance compared to \textsf{DSoG}. However, with \textsf{DSoG} initialization, the two likelihood-based methods become stable. Particularly in scenarios where row and column communities are identical, the performance of likelihood-based methods is superior. This is not only due to the ability of  \textsf{ML-PL} and \textsf{ML-PPL} to handle directed multi-layer networks, but also due to the proper selection of the initial estimator \textsf{DSoG}. In the scenario where row and column communities differ, which is the primary interest in co-clustering, the performance of \textsf{ML-PL} is not satisfactory, suggesting that it may not be suitable for handling directed multi-layer networks without specific adjustments.


\paragraph{\textbf{Efficiency evaluation.}}
In Table \ref{running time}, we present the average running time of compared algorithms on a directed network with 1000 nodes and 50 layers over 100 replications. These simulations were performed on a PC with an Apple M2 CPU processor. The initial values for \textsf{ML-PL} and \textsf{ML-PPL} are set as the result of \textsf{DSoG}. The results show that the proposed \textsf{DSoG} method has the shortest running time, demonstrating its computational efficiency.
\begin{table}[h] 
	\caption{The running time of four methods on a directed network with 1000 nodes and 50 layers.}
	\centering
	\label{running time}
	\begin{tabular}{ccccc}
		\toprule

		Methods & \quad DSoG \quad &  \quad MASE  \quad &  \quad ML-PL  \quad &  \quad ML-PPL \\
		\midrule
		Running time (s) &  \quad6.84  \quad &  \quad 27.18 \quad & \quad 22.63 \quad & \quad 34.59\\

		\bottomrule
	\end{tabular}
\end{table}

\subsubsection{Additional results for real data analysis}
\paragraph{\textbf{The }\textsf{Sum }\textbf{method for WFAT data.}}
Similar to \textsf{DSoG}, we use the scree plot and the Davies-Bouldin score to select the embedding dimension and the number of clusters. We select the embedding dimension as 2, with $K_y = 4$ and $K_z = 5$ for the respective row and column clusters. Regarding the row clustering, the number of clusters identified by the \textsf{Sum} method is lower compared to the \textsf{DSoG} method, which identified 5 clusters. The row clusters estimated by the \textsf{Sum} method are depicted in Figure \ref{map_sum_row}. Compared to the row clusters estimated by the \textsf{DSoG} method in Figure \ref{map}(b), a significant portion of communities 1 and 4 in Figure \ref{map}(b) appear to have merged into one single community, indicating a loss of cluster divisions in the summation process.
\begin{figure*}[!htbp]
    \centering
        \includegraphics[width=0.95\textwidth]{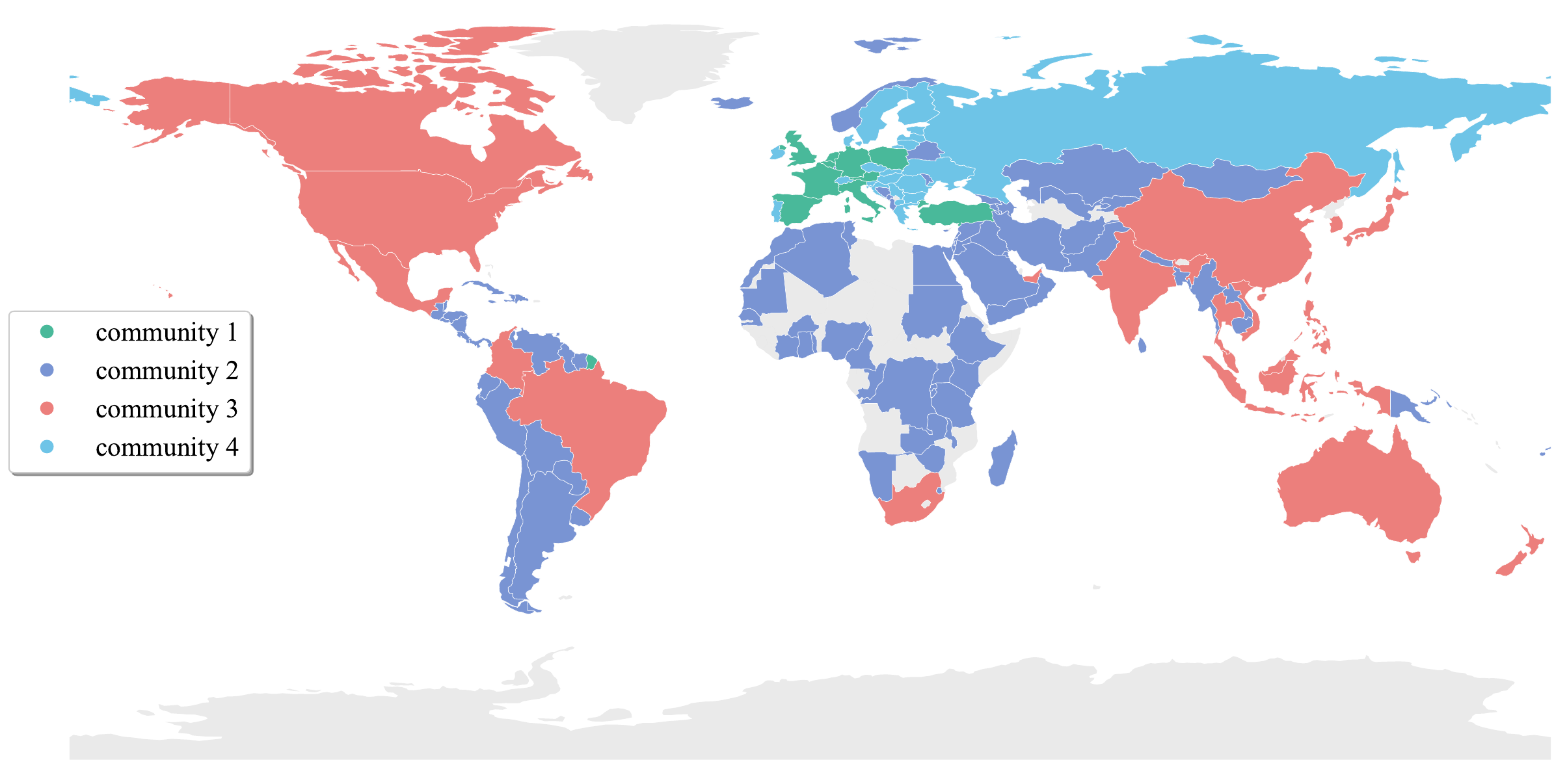}
    \caption{Row community structure separation of food trade networks containing 142 major countries using the \textsf{Sum} method. Colors indicate communities, where light gray corresponds to countries that do not participate in clustering.}
    \label{map_sum_row}
\end{figure*}
\paragraph{\textbf{Block probability matrices quantification.}}
To complement the analysis provided in Figure \ref{Adjs}, we estimate $B_l$'s in our method to show that the single-layer networks alone are not sufficient to provide complete cluster information. Specifically, we estimate $\hat B_l$ by $\hat B_l = \hat Y^T \hat \Delta_y A_l \hat \Delta_z \hat Z$, where $\hat Y$ and $\hat Z$ are the estimated row and column membership matrices, respectively. Both $\hat \Delta_y$ and $\hat \Delta_z$ are $K\times K$ diagonal matrices with the $k$th diagonal entry being the $l_2$ norm of the $k$th column of $\hat Y$ and $\hat Z$, respectively. Figure \ref{Bhat} shows the estimated block  probability matrices, which correspond to the products in Figure  \ref{Adjs}, respectively. It can be seen that for certain layers of networks, the estimated block probability matrices are rank-deficient, indicating that signal for clusters is not strong. For example, for the layer ``Vegetables preserved (frozen)'' (see Figure \ref{Bhat}(a)), the fifth clusters in both the row and column clusters show almost no connections with the other clusters, indicating the vanish of fifth cluster when using this single layer.

\begin{figure*}[h]
    \centering
    \begin{subfigure}[b]{0.32\textwidth}
        \includegraphics[width=\textwidth]{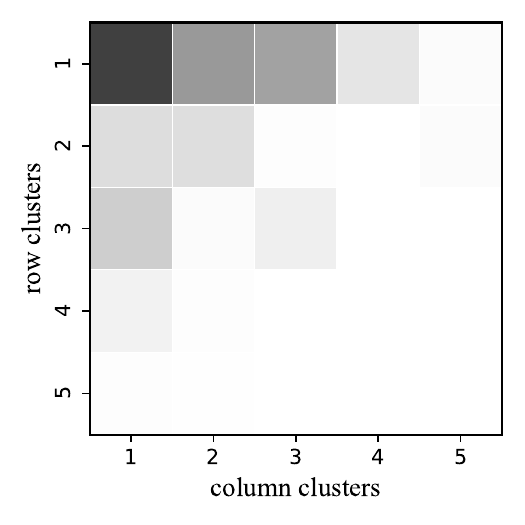}
        \vspace{-0.2\textwidth}
        \caption{Vegetables preserved, frozen}
    \end{subfigure}
    \begin{subfigure}[b]{0.32\textwidth}
        \includegraphics[width=\textwidth]{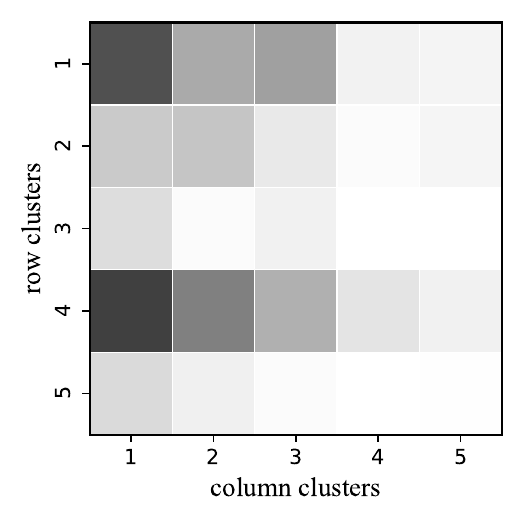}
        \vspace{-0.2\textwidth}
        \caption{Coffee, green}
    \end{subfigure}
    \begin{subfigure}[b]{0.32\textwidth}
        \includegraphics[width=\textwidth]{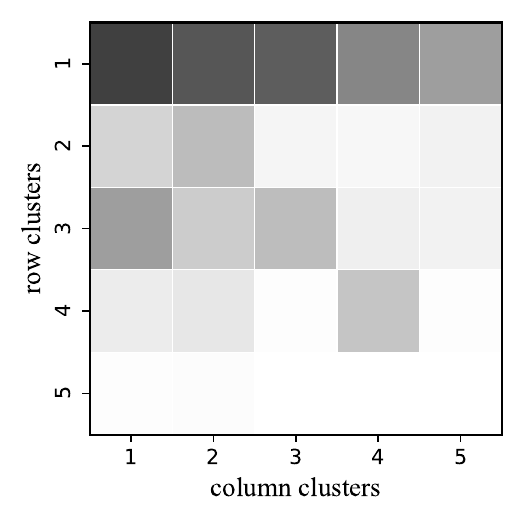}
        \vspace{-0.2\textwidth}
        \caption{Food wastes}
    \end{subfigure} \\
    \vspace{0.025\textwidth}
    \begin{subfigure}[b]{0.32\textwidth}
        \includegraphics[width=\textwidth]{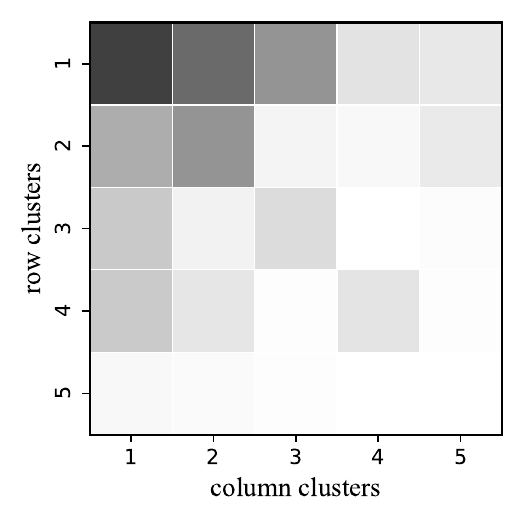}
        \vspace{-0.2\textwidth}
        \caption{Juice of fruits n.e.c.}
    \end{subfigure}
    \begin{subfigure}[b]{0.32\textwidth}
        \includegraphics[width=\textwidth]{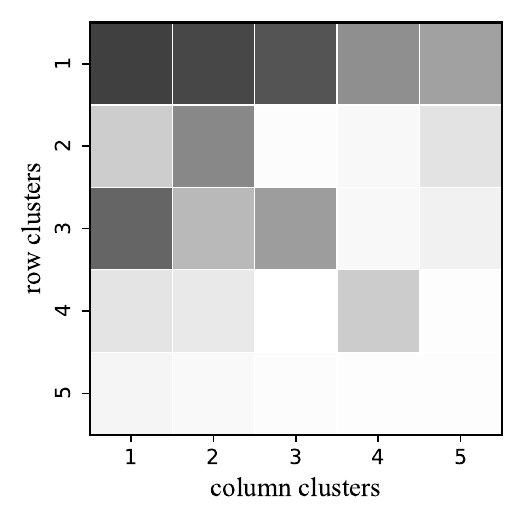}
        \vspace{-0.2\textwidth}
        \caption{Chocolate products nes}
    \end{subfigure}
    \begin{subfigure}[b]{0.32\textwidth}
        \includegraphics[width=\textwidth]{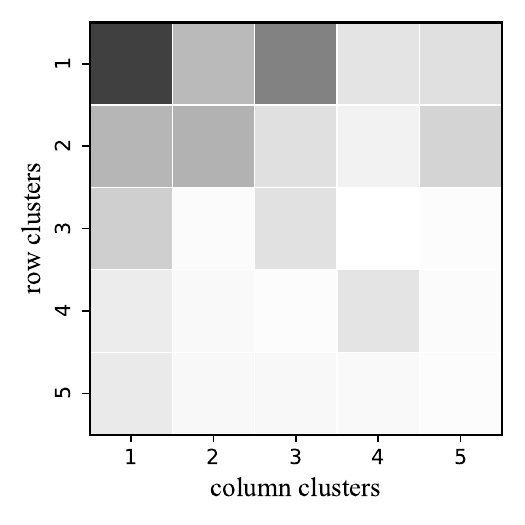}
        \vspace{-0.2\textwidth}
        \caption{Rice, paddy}
    \end{subfigure}
    
    \caption{The estimated block  probability matrices for six products. Darker colors indicate larger values.}
    \label{Bhat}
\end{figure*}

\section*{Acknowledgements}
	The authors are grateful to the editor, associate editor, and two anonymous reviewers for their insightful comments and suggestions.


\bibliographystyle{chicago}
\bibliography{reference}

\begin{thebibliography}{}

\bibitem[\protect\citeauthoryear{Arroyo, Athreya, Cape, Chen, Priebe, and
  Vogelstein}{Arroyo et~al.}{2021}]{arroyo2021inference}
Arroyo, J., A.~Athreya, J.~Cape, G.~Chen, C.~E. Priebe, and J.~T. Vogelstein
  (2021).
\newblock Inference for multiple heterogeneous networks with a common invariant
  subspace.
\newblock {\em Journal of Machine Learning Research\/}~{\em 22\/}(142), 1--49.

\bibitem[\protect\citeauthoryear{Bakken, Miller, Ding, Sunkin, Smith, Ng,
  Szafer, Dalley, Royall, Lemon, et~al.}{Bakken
  et~al.}{2016}]{bakken2016comprehensive}
Bakken, T.~E., J.~A. Miller, S.-L. Ding, S.~M. Sunkin, K.~A. Smith, L.~Ng,
  A.~Szafer, R.~A. Dalley, J.~J. Royall, T.~Lemon, et~al. (2016).
\newblock A comprehensive transcriptional map of primate brain development.
\newblock {\em Nature\/}~{\em 535\/}(7612), 367--375.

\bibitem[\protect\citeauthoryear{Bhatia}{Bhatia}{1997}]{Bhatia1997matrix}
Bhatia, R. (1997).
\newblock {\em Matrix analysis}, Volume 169 of {\em Graduate Texts in
  Mathematics}.
\newblock Springer-Verlag, New York.

\bibitem[\protect\citeauthoryear{Bhattacharyya and Chatterjee}{Bhattacharyya
  and Chatterjee}{2018}]{bhattacharyya2018spectral}
Bhattacharyya, S. and S.~Chatterjee (2018).
\newblock Spectral clustering for multiple sparse networks: I.
\newblock {\em arXiv preprint arXiv:1805.10594\/}.

\bibitem[\protect\citeauthoryear{Boccaletti, Bianconi, Criado, Del~Genio,
  G{\'o}mez-Gardenes, Romance, Sendina-Nadal, Wang, and Zanin}{Boccaletti
  et~al.}{2014}]{boccaletti2014structure}
Boccaletti, S., G.~Bianconi, R.~Criado, C.~I. Del~Genio, J.~G{\'o}mez-Gardenes,
  M.~Romance, I.~Sendina-Nadal, Z.~Wang, and M.~Zanin (2014).
\newblock The structure and dynamics of multilayer networks.
\newblock {\em Physics Reports\/}~{\em 544\/}(1), 1--122.

\bibitem[\protect\citeauthoryear{Chen, Liu, and Ma}{Chen
  et~al.}{2022}]{chen2022global}
Chen, S., S.~Liu, and Z.~Ma (2022).
\newblock Global and individualized community detection in inhomogeneous
  multilayer networks.
\newblock {\em The Annals of Statistics\/}~{\em 50\/}(5), 2664--2693.

\bibitem[\protect\citeauthoryear{Davies and Bouldin}{Davies and
  Bouldin}{1979}]{davies1979cluster}
Davies, D.~L. and D.~W. Bouldin (1979).
\newblock A cluster separation measure.
\newblock {\em IEEE transactions on pattern analysis and machine
  intelligence\/}~{\em PAMI-1\/}(2), 224--227.

\bibitem[\protect\citeauthoryear{De~Domenico, Nicosia, Arenas, and
  Latora}{De~Domenico et~al.}{2015}]{de2015structural}
De~Domenico, M., V.~Nicosia, A.~Arenas, and V.~Latora (2015).
\newblock Structural reducibility of multilayer networks.
\newblock {\em Nature Communications\/}~{\em 6}, 6864.

\bibitem[\protect\citeauthoryear{de~la Pe{\~n}a and Montgomery-Smith}{de~la
  Pe{\~n}a and Montgomery-Smith}{1995}]{de1995decoupling}
de~la Pe{\~n}a, V.~H. and S.~J. Montgomery-Smith (1995).
\newblock Decoupling inequalities for the tail probabilities of multivariate
  u-statistics.
\newblock {\em The Annals of Probability\/}~{\em 23\/}(2), 806--816.

\bibitem[\protect\citeauthoryear{Della~Rossa, Pecora, Blaha, Shirin,
  Klickstein, and Sorrentino}{Della~Rossa et~al.}{2020}]{della2020symmetries}
Della~Rossa, F., L.~Pecora, K.~Blaha, A.~Shirin, I.~Klickstein, and
  F.~Sorrentino (2020).
\newblock Symmetries and cluster synchronization in multilayer networks.
\newblock {\em Nature Communications\/}~{\em 11}, 3179.

\bibitem[\protect\citeauthoryear{Fishkind, Sussman, Tang, Vogelstein, and
  Priebe}{Fishkind et~al.}{2013}]{fishkind2013consistent}
Fishkind, D.~E., D.~L. Sussman, M.~Tang, J.~T. Vogelstein, and C.~E. Priebe
  (2013).
\newblock Consistent adjacency-spectral partitioning for the stochastic block
  model when the model parameters are unknown.
\newblock {\em SIAM Journal on Matrix Analysis and Applications\/}~{\em
  34\/}(1), 23--39.

\bibitem[\protect\citeauthoryear{Fu and Hu}{Fu and Hu}{2023}]{fu2023profile}
Fu, K. and J.~Hu (2023).
\newblock Profile-pseudo likelihood methods for community detection of
  multilayer stochastic block models.
\newblock {\em Stat\/}~{\em 12\/}(1), e594.

\bibitem[\protect\citeauthoryear{Golub and Van~Loan}{Golub and
  Van~Loan}{2013}]{golub2013matrix}
Golub, G.~H. and C.~F. Van~Loan (2013).
\newblock {\em Matrix computations}.
\newblock JHU press, Baltimore.

\bibitem[\protect\citeauthoryear{Guo, Qiu, Zhang, and Chang}{Guo
  et~al.}{2023}]{guo2023randomized}
Guo, X., Y.~Qiu, H.~Zhang, and X.~Chang (2023).
\newblock Randomized spectral co-clustering for large-scale directed networks.
\newblock {\em Journal of Machine Learning Research\/}~{\em 24\/}(380), 1--68.

\bibitem[\protect\citeauthoryear{Han, Xu, and Airoldi}{Han
  et~al.}{2015}]{han2015consistent}
Han, Q., K.~Xu, and E.~Airoldi (2015).
\newblock Consistent estimation of dynamic and multi-layer block models.
\newblock In {\em International Conference on Machine Learning}, pp.\
  1511--1520. PMLR.

\bibitem[\protect\citeauthoryear{Holland, Laskey, and Leinhardt}{Holland
  et~al.}{1983}]{holland1983stochastic}
Holland, P.~W., K.~B. Laskey, and S.~Leinhardt (1983).
\newblock Stochastic blockmodels: First steps.
\newblock {\em Social Networks\/}~{\em 5\/}(2), 109--137.

\bibitem[\protect\citeauthoryear{Holme and Saram{\"a}ki}{Holme and
  Saram{\"a}ki}{2012}]{holme2012temporal}
Holme, P. and J.~Saram{\"a}ki (2012).
\newblock Temporal networks.
\newblock {\em Physics Reports\/}~{\em 519\/}(3), 97--125.

\bibitem[\protect\citeauthoryear{Hu, Qin, Yan, and Zhao}{Hu
  et~al.}{2020}]{hu2020corrected}
Hu, J., H.~Qin, T.~Yan, and Y.~Zhao (2020).
\newblock Corrected bayesian information criterion for stochastic block models.
\newblock {\em Journal of the American Statistical Association\/}~{\em
  115\/}(532), 1771--1783.

\bibitem[\protect\citeauthoryear{Huang, Weng, and Feng}{Huang
  et~al.}{2023}]{huang2022spectral}
Huang, S., H.~Weng, and Y.~Feng (2023).
\newblock Spectral clustering via adaptive layer aggregation for multi-layer
  networks.
\newblock {\em Journal of Computational and Graphical Statistics\/}~{\em
  32\/}(3), 1170--1184.

\bibitem[\protect\citeauthoryear{Jing, Li, Lyu, and Xia}{Jing
  et~al.}{2021}]{jing2021community}
Jing, B.-Y., T.~Li, Z.~Lyu, and D.~Xia (2021).
\newblock Community detection on mixture multilayer networks via regularized
  tensor decomposition.
\newblock {\em The Annals of Statistics\/}~{\em 49\/}(6), 3181--3205.

\bibitem[\protect\citeauthoryear{Kivel{\"a}, Arenas, Barthelemy, Gleeson,
  Moreno, and Porter}{Kivel{\"a} et~al.}{2014}]{kivela2014multilayer}
Kivel{\"a}, M., A.~Arenas, M.~Barthelemy, J.~P. Gleeson, Y.~Moreno, and M.~A.
  Porter (2014).
\newblock Multilayer networks.
\newblock {\em Journal of Complex Networks\/}~{\em 2\/}(3), 203--271.

\bibitem[\protect\citeauthoryear{Kumar, Sabharwal, and Sen}{Kumar
  et~al.}{2004}]{kumar2004simple}
Kumar, A., Y.~Sabharwal, and S.~Sen (2004).
\newblock A simple linear time (1 + $\varepsilon$)-approximation algorithm for
  $k$-means clustering in any dimensions.
\newblock In {\em 45th Annual IEEE Symposium on Foundations of Computer
  Science}, pp.\  454--462. IEEE.

\bibitem[\protect\citeauthoryear{Lei, Chen, and Lynch}{Lei
  et~al.}{2020}]{lei2020consistent}
Lei, J., K.~Chen, and B.~Lynch (2020).
\newblock Consistent community detection in multi-layer network data.
\newblock {\em Biometrika\/}~{\em 107\/}(1), 61--73.

\bibitem[\protect\citeauthoryear{Lei and Lin}{Lei and Lin}{2023}]{LeiJ2022Bias}
Lei, J. and K.~Z. Lin (2023).
\newblock Bias-adjusted spectral clustering in multi-layer stochastic block
  models.
\newblock {\em Journal of the American Statistical Association\/}~{\em
  118\/}(544), 2433--2445.

\bibitem[\protect\citeauthoryear{Lei and Rinaldo}{Lei and
  Rinaldo}{2015}]{lei2015consistency}
Lei, J. and A.~Rinaldo (2015).
\newblock Consistency of spectral clustering in stochastic block models.
\newblock {\em The Annals of Statistics\/}~{\em 43\/}(1), 215--237.

\bibitem[\protect\citeauthoryear{Li, Levina, and Zhu}{Li
  et~al.}{2020}]{li2020network}
Li, T., E.~Levina, and J.~Zhu (2020).
\newblock Network cross-validation by edge sampling.
\newblock {\em Biometrika\/}~{\em 107\/}(2), 257--276.

\bibitem[\protect\citeauthoryear{Ma, Su, and Zhang}{Ma
  et~al.}{2021}]{ma2021determining}
Ma, S., L.~Su, and Y.~Zhang (2021).
\newblock Determining the number of communities in degree-corrected stochastic
  block models.
\newblock {\em Journal of Machine Learning Research\/}~{\em 22\/}(310), 1--61.

\bibitem[\protect\citeauthoryear{MacDonald, Levina, and Zhu}{MacDonald
  et~al.}{2022}]{macdonald2022latent}
MacDonald, P.~W., E.~Levina, and J.~Zhu (2022).
\newblock Latent space models for multiplex networks with shared structure.
\newblock {\em Biometrika\/}~{\em 109\/}(3), 683--706.

\bibitem[\protect\citeauthoryear{Malliaros and Vazirgiannis}{Malliaros and
  Vazirgiannis}{2013}]{malliaros2013clustering}
Malliaros, F.~D. and M.~Vazirgiannis (2013).
\newblock Clustering and community detection in directed networks: A survey.
\newblock {\em Physics Reports\/}~{\em 533\/}(4), 95--142.

\bibitem[\protect\citeauthoryear{Mucha, Richardson, Macon, Porter, and
  Onnela}{Mucha et~al.}{2010}]{mucha2010community}
Mucha, P.~J., T.~Richardson, K.~Macon, M.~A. Porter, and J.-P. Onnela (2010).
\newblock Community structure in time-dependent, multiscale, and multiplex
  networks.
\newblock {\em Science\/}~{\em 328\/}(5980), 876--878.

\bibitem[\protect\citeauthoryear{Noroozi and Pensky}{Noroozi and
  Pensky}{2022}]{noroozi2022sparse}
Noroozi, M. and M.~Pensky (2022).
\newblock Sparse subspace clustering in diverse multiplex network model.
\newblock {\em arXiv preprint arXiv:2206.07602\/}.

\bibitem[\protect\citeauthoryear{Paul and Chen}{Paul and
  Chen}{2016}]{paul2016consistent}
Paul, S. and Y.~Chen (2016).
\newblock Consistent community detection in multi-relational data through
  restricted multi-layer stochastic blockmodel.
\newblock {\em Electronic Journal of Statistics\/}~{\em 10\/}(2), 3807--3870.

\bibitem[\protect\citeauthoryear{Paul and Chen}{Paul and
  Chen}{2020}]{paul2020spectral}
Paul, S. and Y.~Chen (2020).
\newblock Spectral and matrix factorization methods for consistent community
  detection in multi-layer networks.
\newblock {\em The Annals of Statistics\/}~{\em 48\/}(1), 230--–250.

\bibitem[\protect\citeauthoryear{Pensky and Wang}{Pensky and
  Wang}{2021}]{pensky2021clustering}
Pensky, M. and Y.~Wang (2021).
\newblock Clustering of diverse multiplex networks.
\newblock {\em arXiv preprint arXiv:2110.05308\/}.

\bibitem[\protect\citeauthoryear{Pensky and Zhang}{Pensky and
  Zhang}{2019}]{pensky2019spectral}
Pensky, M. and T.~Zhang (2019).
\newblock Spectral clustering in the dynamic stochastic block model.
\newblock {\em Electronic Journal of Statistics\/}~{\em 13\/}(1), 678--709.

\bibitem[\protect\citeauthoryear{Rohe, Qin, and Yu}{Rohe
  et~al.}{2016}]{rohe2016co}
Rohe, K., T.~Qin, and B.~Yu (2016).
\newblock Co-clustering directed graphs to discover asymmetries and directional
  communities.
\newblock {\em Proceedings of the National Academy of Sciences\/}~{\em
  113\/}(45), 12679--12684.

\bibitem[\protect\citeauthoryear{Tang, Cape, and Priebe}{Tang
  et~al.}{2022}]{tang2022asymptotically}
Tang, M., J.~Cape, and C.~E. Priebe (2022).
\newblock Asymptotically efficient estimators for stochastic blockmodels: The
  naive mle, the rank-constrained mle, and the spectral estimator.
\newblock {\em Bernoulli\/}~{\em 28\/}(2), 1049--1073.

\bibitem[\protect\citeauthoryear{Valles-Catala, Massucci, Guimera, and
  Sales-Pardo}{Valles-Catala et~al.}{2016}]{valles2016multilayer}
Valles-Catala, T., F.~A. Massucci, R.~Guimera, and M.~Sales-Pardo (2016).
\newblock Multilayer stochastic block models reveal the multilayer structure of
  complex networks.
\newblock {\em Physical Review X\/}~{\em 6\/}(1), 011036.

\bibitem[\protect\citeauthoryear{Vu and Lei}{Vu and Lei}{2013}]{vu2013minimax}
Vu, V.~Q. and J.~Lei (2013).
\newblock Minimax sparse principal subspace estimation in high dimensions.
\newblock {\em The Annals of Statistics\/}~{\em 41\/}(6), 2905--2947.

\bibitem[\protect\citeauthoryear{Wang, Guo, and Liu}{Wang
  et~al.}{2021}]{wang2021fast}
Wang, J., J.~Guo, and B.~Liu (2021).
\newblock A fast algorithm for integrative community detection of multi-layer
  networks.
\newblock {\em Stat\/}~{\em 10\/}(1), e348.

\bibitem[\protect\citeauthoryear{Zhang and Cao}{Zhang and
  Cao}{2017}]{zhang2017finding}
Zhang, J. and J.~Cao (2017).
\newblock Finding common modules in a time-varying network with application to
  the drosophila melanogaster gene regulation network.
\newblock {\em Journal of the American Statistical Association\/}~{\em
  112\/}(519), 994--1008.

\bibitem[\protect\citeauthoryear{Zhang, He, and Wang}{Zhang
  et~al.}{2022}]{zhang2022directed}
Zhang, J., X.~He, and J.~Wang (2022).
\newblock Directed community detection with network embedding.
\newblock {\em Journal of the American Statistical Association\/}~{\em
  117\/}(540), 1809--1819.

\end{thebibliography}
\end{document}